\documentclass[reqno,a4paper,12pt]{amsart}

\usepackage[
  hmarginratio={1:1},     
  vmarginratio={1:1},     
  textwidth=148mm,        
  textheight=220mm,       
  marginparwidth=27mm,
  marginparsep=3mm
]{geometry}

\usepackage[no-math]{fontspec}
\usepackage[cachedir=minted-cache]{minted}

\newmintinline[lean]{lean4}{bgcolor=white}
\newminted[leancode]{lean4}{escapeinside=!!,
                            breaklines,
                            fontsize=\normalsize}
\numberwithin{equation}{section}
\usepackage{amsmath, amsthm, amssymb, amscd, accents, bm, amsfonts}
\usepackage{mathtools}
\usepackage{url}
\usepackage{mathrsfs,dsfont}
\usepackage{mathtools}
\usepackage{datetime}
\usepackage{hyperref}
\usepackage{colonequals}
\usepackage{enumerate}
\usepackage{xcolor}
\usepackage[sort,nocompress]{cite}

\mathtoolsset{showonlyrefs}

\newtheorem{definition}{Definition}[section]
\newtheorem{theorem}[definition]{Theorem}
\newtheorem*{theorem*}{Theorem}
\newtheorem{lemma}[definition]{Lemma}

\newtheorem{proposition}[definition]{Proposition}

\newtheorem{remark}[definition]{Remark}

\def\N{{\mathbb N}}
\def\Z{{\mathbb Z}}
\def\R{{\mathbb R}}

\def\C{{\mathbb C}}
\def\Q{{\mathbb Q}}

\newcommand{\supp}{{\mathrm{supp}}}
\newcommand{\lt}{{L^2(\R)}}

\newcommand{\ft}{{\mathcal{F}}}

\allowdisplaybreaks

\newcommand{\norm}[1]{\left\lVert #1 \right\rVert}
\newcommand{\abs}[1]{\left\lvert #1 \right\rvert}
\newcommand{\dist}{\operatorname{dist}}

\newcommand{\cG}{\mathcal{G}}
\NewDocumentCommand{\preG}{ o o }{%
  \IfNoValueTF{#1}{%
    G
  }{%
    \IfNoValueTF{#2}{%
      G_{#1}
    }{%
      G_{{#1(#2)}}
    }%
  }%
}

\begin{document}

\title[Gabor Frames of Totally Positive Functions]{Gabor Frames  of
  Totally Positive Functions: \\ A Complete Characterization}

\author[J. de Dios Pont]{Jaume de Dios Pont}
\address{Center for Data Science, New York University, New York, New York 10011, USA}
\email{jdedios@nyu.edu}

\author[K. Gröchenig]{Karlheinz Gröchenig}
\address{Faculty of Mathematics, University of Vienna, Oskar-Morgenstern-Platz 1, A-1090 Vienna, Austria}
\email{karlheinz.groechenig@univie.ac.at}

\author[L. Liehr]{Lukas Liehr}
\address{Department of Mathematics, Bar-Ilan University, Ramat-Gan 5290002, Israel}
\email{lukas.liehr@biu.ac.il}

\author[I. Shafkulovska]{Irina Shafkulovska}
\address{Faculty of Mathematics, University of Vienna, Oskar-Morgenstern-Platz 1, A-1090 Vienna, Austria}
\email{irina.shafkulovska@univie.ac.at}

\author[M. A. Taylor]{Mitchell A. Taylor}
\address{ Department of Mathematics, ETH Z\"urich, Ramistrasse 101, 8092 Z\"urich, Switzerland}
\email{mitchell.taylor@math.ethz.ch}

\date{\today}

\begin{abstract}
We prove that the set of time-frequency shifts
$$\left \{ e^{2\pi i \beta lt }g(t-\alpha k) : k,l\in \Z \right \}$$ with a continuous,
integrable  totally positive function $g$  and lattice parameters $\alpha ,\beta
>0$ generates a frame for $L^2(\R)$ if and
only if $\alpha \beta <1$. This fully settles the so-called  frame set problem for the
class of totally positive functions. As a closely related result we
prove a  sharp Kadets-type theorem for every shift-invariant space
generated by  a continuous totally positive function. The proofs are  based on
Fredholm theory and limit-operator theory. A formalization of our main
 result in Lean 4 is also provided. 
\end{abstract}

\maketitle

\section{Introduction and results}

\subsection{Gabor frames}
For a measurable function $g: \R \to \C$ and positive scalars
$\alpha,\beta>0$, the Gabor system $\cG(g,\alpha,\beta)$ is defined by
\begin{equation}\label{eq:def:cG}
    \cG(g,\alpha,\beta)=\left \{ e^{2\pi i \beta lt }g(t-\alpha k) :k,l\in \Z \right \}.
\end{equation}
The set  $\cG(g,\alpha,\beta)$ consists of a structured collection of time-frequency shifts of $g$, i.e., each element of $\cG(g,\alpha,\beta)$ arises from an application of the translation operator $T_x f(t) = f(t-x)$ and the modulation operator $M_\xi f(t) = e^{2\pi i \xi t} f(t)$.
If $g\in L^2(\R)$, we say that $\cG(g,\alpha,\beta)$ is a \emph{Gabor
frame} if there exist constants $0<A\leq B<\infty$ such that for every $f \in \lt$ one has
$$
A\norm{f}^2
\leq
\sum_{k,l\in \Z}\abs{\langle f,M_{\beta l}T_{\alpha k}g\rangle}^2 \leq
B\norm{f}^2,
$$
where $\langle \cdot ,\cdot \rangle$ denotes the $L^2$-inner product
and $\| \cdot  \|$ the $L^2$-norm.

Note that the lower frame inequality implies  that the
set $\{ e^{2\pi i \beta lt }g(t-\alpha k) :k,l\in \Z  \}$ is
complete in $L^2(\R )$. Questions about the spanning properties of
time-frequency shifts go back to von Neumann's ``Foundations of
Quantum Mechanics''~\cite{neumann} and to  Gabor's ``Theory of
Communication''~\cite{gabor}. More recently, Gabor frames have played an
important role in the processing of acoustic and audio signals, and
the frame property, as opposed to completeness, has been brought
center-stage by Daubechies' influential
work~\cite{daubechies2002wavelet}.

A central topic in time-frequency analysis is the characterization of
those $g$ and $\alpha,\beta$ for which $\cG(g,\alpha,\beta)$ forms a
frame \cite{Groechenig2001,heil2007history,mystery}. This problem amounts to determining the \emph{frame set} of $g$ which is defined by
$$
\ft(g) = \left \{ (\alpha,\beta) \in \R_+^2 : \cG(g,\alpha,\beta) \ \text{is a frame} \right \} ,
$$
where $\R_+ = \{ x \in \R : x > 0 \}$. The density theorem for Gabor
frames asserts that $\ft (g) \subseteq \{ (\alpha,\beta) \in \R_{+}^2
: \alpha\beta \leq  1 \}$ for all $g\in L^2(\R )$ and, with a harmless
assumption on the decay and smoothness of $g$, one even knows that
\begin{equation}  \label{eq:c5}
\ft (g) \subseteq \{ (\alpha,\beta) \in \R_{+}^2 : \alpha\beta < 1
\}\, .  
\end{equation}
See~\cite{heil2007history} for a comprehensive discussion of the
density theorem and the Balian-Low theorem for Gabor frames. 

To this day, the frame set has been determined for  only a few classes
of functions. Of particular importance was 
an influential conjecture of Daubechies saying that if both $g$ and
its Fourier transform are positive, then the Gabor system
$\mathcal{G}(g,\alpha, \beta)$ is a frame whenever $\alpha\beta<1$
\cite[p.~981]{daubechies2002wavelet}. In other words,  the frame set is maximal,
and the necessary condition~\eqref{eq:c5}  is  also sufficient.  Although this
conjecture was shown to be false in general by Janssen
\cite{janssen2002some}, it initiated extensive research on identifying
function classes where the frame set  equals the maximal  set $\{
(\alpha,\beta) \in \R_{+}^2 : \alpha\beta < 1 \}$, which we will
abbreviate by $\mathcal{M} = \{
(\alpha,\beta) \in \R_{+}^2 : \alpha\beta < 1 \}$ 
\cite{belov2023gabor,Groechenig2023,grochenig2018sampling,GroechenigStoeckler2013,Janssen1996,janssen2002hyperbolic}.

\subsection{Totally positive functions}
An important step towards the deeper understanding of the frame set  was the insight from ~\cite{GroechenigStoeckler2013} that the appropriate
notion of positivity in the study of Gabor frames was \emph{total
  positivity} of  functions. A function $g:\R\to \R$ is called
\emph{totally positive} if it is not constant and for every 
$n\in \N$ and every
$x_1<\dots<x_n$ and
$y_1<\dots<y_n$, the determinant of the matrix $M \in \R^{n \times n}$
with entries $M_{jk} = g(x_j-y_k)$ satisfies $\det M \geq 0$.  Every
integrable  totally positive function  has exponential decay
\cite{Schoenberg1947}. Assuming also continuity,  their frame set is
then  automatically  contained in the open set $\mathcal{M}$.

Totally positive functions occur in many branches of mathematics,
e.g., spline theory, statistics, the representation of
infinite-dimensional groups, and also in sampling theory and
time-frequency analysis. General references are the books by Karlin
\cite{karlin1968total} and  Pinkus
\cite{pinkus2010totally}, and a survey of their role in time-frequency
analysis and sampling is contained in ~\cite{grosurvey22}. 

\subsection{Main result}
The study of frame sets for totally positive functions has been the
subject of a series of papers
\cite{grochenig2018sampling,Groechenig2023,GroechenigStoeckler2013,Janssen1996,Seip1992,lyubarskiiframes}
and draws  on techniques from several areas of mathematics, including harmonic analysis, functional analysis, and complex analysis. Gabor systems associated with totally positive functions were also investigated in several related contexts \cite{bannert2013discretized,grochenig2016completeness,kloos2016implementation,ghosh2025gabor}, including discrete settings, completeness problems, algorithmic aspects, and quantitative estimates for frame bounds.
Central to this line of research is the following conjecture 
in~\cite{mystery,GroechenigStoeckler2013}: for every continuous
totally positive function in $L^1(\R)$, the frame set coincides with
$\mathcal{M} = \left \{ (\alpha,\beta) \in \R_{+}^2 : \alpha\beta < 1  \right
\}$. This statement is often  referred to as the \emph{frame set
  conjecture for totally positive functions}. 

In this  paper, we prove the frame set conjecture for totally positive functions in full generality.

\begin{theorem}\label{thm:main}
The frame set of every continuous totally positive function $g \in
L^1(\R) $ is given by $\ft(g) = \left \{ (\alpha,\beta) \in \R_{+}^2 :
  \alpha\beta < 1  \right \}$. This means that $\mathcal{G}
(g,\alpha,\beta )$ is a frame if and only if $\alpha \beta <1$. 
\end{theorem}

Note that the only discontinuous totally positive functions in
$L^1(\R)$ are one-sided exponentials, for which the frame set was
determined in early work by Janssen \cite{Janssen1996} to be $\ft (g)
= \{
(\alpha,\beta) \in \R_{+}^2 : \alpha\beta \leq 1 \}$.  

\subsection{Previous work}
Theorem~\ref{thm:main} has a long history. It  has been established
under various additional assumptions, leading to a number of partial
results in several important special cases.
The first fundamental result is the determination of the frame set for
the Gaussian $\varphi(t) = e^{-\pi t^2}$  as $\ft(\varphi) =
\mathcal{M} $ by  Seip-Wallstén and
Lyubarskii~\cite{lyubarskiiframes,Seip1992,seip-wallsten}. The
proof is based on the  translation into an equivalent  sampling
problem in Fock space and techniques for entire functions. In fact,
Seip-Wallstén and Lyubarskii proved a more general result for
arbitrary sets of time-frequency shifts.

In the sequel, a few  other isolated  examples of functions with the maximal
frame set $\left \{  (\alpha,\beta) \in \R_{+}^2 : \alpha\beta < 1
\right \}$   were found, namely the one-sided exponential, the
exponential function $e^{-c|x|}$~\cite{Janssen1996}, the hyperbolic secant $(e^{cx}+
e^{-cx})^{-1}$~\cite{janssen2002hyperbolic}, and their Fourier transforms.  It turned out that all
these  examples were  totally positive functions. 

This observation led to an intensive investigation of Gabor frames
based on totally positive functions. In this context,  Schoenberg's
factorization theorem helped to identify several subclasses of totally
positive functions for which the frame set could be
characterized. Schoenberg's theorem \cite{Schoenberg1951} asserts that
$g \in L^1(\R)$ is totally positive if and only if its  Fourier transform $\hat g$ admits a factorization of the form
\begin{equation}\label{eq:schoenberg}
    \hat{g}(\xi)
=
c e^{-\gamma \xi^2} e^{2\pi i \nu \xi}
\prod_{j=1}^{N}
\frac{e^{-2\pi i \nu_j \xi}}{1 + 2\pi i \nu_j \xi},
\end{equation}
where $c>0$, $\nu,\nu_j \in \mathbb{R}$, $\gamma \geq 0$, $N \in \mathbb{N}_0\cup\{\infty\}$ and
$
0 < \gamma + \sum_j \nu_j^2 < \infty.
$

The frame set of a totally positive function $g$ is the maximal set
$\ft(g) = \mathcal{M}$, whenever the parameters of the factorization satisfy either 

(i) $\gamma = 0$ and the product \eqref{eq:schoenberg} is finite and
$N\geq 2$~\cite{GroechenigStoeckler2013}, or

(ii) $\gamma >0$ and the product is finite $N\in \N
_0$~\cite{grochenig2018sampling}.

We remark that the proofs for these two cases required completely
different methods.  For (i), the definition of total positivity and the
Schoenberg-Whitney conditions were used, and  for (ii) a zero-counting
argument from complex analysis was central. 

For arbitrary totally positive functions, a dense open subset of
$\mathcal{M}$
was shown to be contained in the frame set. Precisely,  if $g \in L^1(\R)$
is a continuous totally positive function and if $\alpha,\beta \in
\R_+$ satisfy $\alpha\beta < 1$ and $\alpha\beta \in \Q$, then
$\cG(g,\alpha,\beta)$ is a frame \cite{Groechenig2023}. The proof of
this result bypasses Schoenberg's factorization and works directly
with total positivity.

The case of irrational $\alpha \beta $ remained open. It is this case
that Theorem~\ref{thm:main} addresses and solves completely.  

\subsection{Proof technique}
The proof of Theorem~\ref{thm:main} follows the outline of
\cite{Groechenig2023}. It suffices to extract a subsequence $\{ k+
\delta _k: k\in \Z\}$ from every $x+\alpha \Z$ and show that every 
restricted pre-Gramian $G_\delta$ with entries
\begin{equation}
  \label{eq:c3}
(G_\delta)_{kl} =  g(k+\delta _k - l) \qquad k,l  \in \Z  
\end{equation}
is invertible on either $\ell ^1(\Z )$ or $\ell ^\infty (\Z )$. 
Using two substantial results about totally positive matrices and
functions, namely the 
de Boor-Friedland-Pinkus criterion for infinite totally positive
matrices~\cite{deBoorFriedlandPinkus1982} and the fact that the Zak
transform of a totally positive function has a  unique zero in its
fundamental domain~\cite{VinogradovUlitskaya2017}, one can show that
$G _\delta $ is surjective.

The crucial step is   to prove  that $G _\delta $ is also injective. 
If  $\alpha\beta$ is rational, then $G_\delta $ has some periodicity
property 
and can be identified with a finite-dimensional quadratic  matrix-function\cite{zibulskiZeevi1997analysis,zeeviZibulskiPorat1998multiwindow}.  In
finite dimensions, surjectivity implies injectivity. 

If $\alpha \beta $ is irrational, then the pre-Gramian has less
structure, and new ideas are required to prove injectivity. Dealing
with the kernel and cokernel of an operator suggests to use   Fredholm
theory and the techniques of limit-operators as appropriate
tools~\cite{lindner2006infinite,rabinovitch2004limit,seidel2015semi}.  

To explain the idea, let  $T_m$ denote the shift by $m\in \Z $ acting  on
$\ell^1(\mathbb Z)$. Then an operator $B$ is called a limit operator of the operator $A$ if there exists an unbounded sequence $(m_n)_{n \in \N} \subset \Z$ such that
$$
T_{-m_n} A T_{m_n} \to B, \quad |m_n|\to\infty,
$$
with convergence in a suitable topology, namely the $\mathcal
P$-topology. The collection of all limit operators of $A$ is called
the operator spectrum $\sigma_{\mathrm{op}}(A)$.
Moreover, $A$ is said to be \emph{self-similar}  if
$$
A\in\sigma_{\mathrm{op}}(A).
$$
The idea of the proof is based on the observation that
$\mathrm{ker}\,T_{-n}AT_{n} = T_{-n} \, \mathrm{ker}\, A$. If this kernel  
is finite-dimensional, then $T_{-n} \,  \mathrm{ker} \, A \to 0$ in the
$\mathcal{P}$-topology, so it is plausible that the limit operator is
injective. If $A$ is itself a limit operator, then $A$ is injective.

The technical details for these arguments can be extracted from the
work of  Seidel~\cite{seidel2015semi}
on semi-Fredholm band-dominated operators (limits of band
operators). In fact, one of the added benefits of our presentation is
a streamlined proof of Seidel's result for the Banach space $\ell
^1(\Z )$. We believe that the limit-operator techniques will be useful
for other problems in time-frequency analysis and we have therefore tried
to make the presentation as transparent as possible. 

\subsection{A Kadets-type theorem}
In an attempt  to extend the result about the invertibility of
the pre-Gramian, we were  led to a strong
perturbation for sampling in shift-invariant spaces and introduced
additional tools from Fredholm theory. Recall that 
for $g\in L^2(\mathbb R)$, the shift-invariant space $V^2(g) \subset
\lt$ is defined as 
$$
    V^2(g)
    =
    \Big\{ \sum_{l\in\Z} c_l g(\cdot-l) \in L^2(\R): c\in \ell^2(\Z)\Big\} .
$$
We always assume that $\|f\| \asymp \|c\|_2$ for all $f\in V^2(g)$,
which holds under a mild assumption on $g$ and holds, in particular,  for
totally positive functions in $L^1(\R )$~\cite{GroechenigStoeckler2013}.

A set $X \subset \R$ is called a sampling set if there exist
constants $A,B>0$ such that the sampling inequalities
$$
A\|f\|_2^2 \leq \sum _{x\in X} |f(x)|^2 \leq B\|f\|_2^2
$$
hold for all $ f\in V^2(g) $.  
The connection to the theory of Gabor frames is the following
characterization that is underlying many proofs.  Under a mild
condition on the decay and continuity of $g$, the set  $\mathcal{G}(g,\alpha,
1)$ \emph{is a frame for $L^2(\R )$ if and only if $x+\alpha\Z $ is a
  sampling set for $V^2(g)$ for all $x\in [0,\alpha ]$.} This
characterization is implicitly used in our work via
Proposition~\ref{prop:submatrix}. See Section~\ref{sec:52} for more details. 

In the context of shift-invariant spaces,  the pre-Gramian $G_\delta $ from \eqref{eq:c3} is
invertible if and only if the set $\{ j+\delta _j: j\in \Z \}$ is a
sampling set and the interpolation problem $f(j+\delta _j) = y_j$ can
be solved (uniquely)  for every $y\in \ell ^2(\Z )$. Such a sampling
set is called a \emph{complete interpolating set}. In ~\cite{Gro26} it was
suggested that Theorem~\ref{thm:main} could be proved by verifying a
sharp perturbation result for complete interpolating sets in $V^2(g)$
for totally positive generator. This is accomplished by the following
result.

\begin{theorem}\label{thm:kadets}
Let $g\in L^1(\mathbb R)$ be a continuous totally positive function, and let $(x_0,\frac12)\in[0,1)^2$ be the unique zero of its Zak
transform on the unit square. If $(\delta_k)_{k\in\mathbb Z} \subset \R$ is a sequence satisfying
$$
   \sup _{k\in \Z} |\delta_k -(x_0-\tfrac{1}{2})| <\frac12, 
$$
then $X = \{ k+\delta_k : k \in \Z \}$ is a complete interpolating
sequence for $V^2(g)$. Moreover, the constant $\frac12$ is sharp. 
\end{theorem}
The  proof is based on properties of the Fredholm index and an old
result about uniform sampling in shift-invariant spaces by Janssen and
Walter~\cite{janssen94a,walter94}.

We note that Theorem~\ref{thm:kadets} implies
Theorem~\ref{thm:main}. 
Theorem~\ref{thm:kadets} is of interest in itself and  should be
compared to Kadets'  corresponding
statement for  the classical Paley-Wiener space $PW := V^2(g) = \{f\in
L^2(\R): \mathrm{supp} \, \hat{f} \subseteq [-1/2,1/2]\}$ where $g(t)
= \frac{\sin \pi t}{\pi t}$ is the cardinal sine function. In this
case, the integers $\Z $ form a complete interpolating set for $PW$ and
Kadets' classical $\frac14$-theorem~\cite{Kadec1964} asserts that if $\|\delta\|_\infty<\frac14$, then  $\{k+\delta _k: k\in \Z\} $ is
a complete interpolating sequence.
In this case,  the constant $\frac14$ is sharp.

Similar to the results in  \cite{BBG22}, there is a subtle and surprising  difference between the
Paley-Wiener space $PW$ (maximum perturbation  is $<1/4$) and
shift-invariant spaces based on totally positive functions (with
constant $1/2$). 

\subsection{Outline} In Section~2 we collect the key facts about the
Zak transform of totally positive functions and state the required
results from ~\cite{Groechenig2023}. Section~3 deals with the tools
from Fredholm theory and limit-operators that will prove the
injectivity of the pre-Gramian matrix. Section~4 contains the proof of
the main theorem (Theorem~\ref{thm:main}), and Section~5 discusses and
proves the Kadets-type theorem (Theorem~\ref{thm:kadets}) for
shift-invariant spaces.  

\subsection*{Usage of Large Language Models}

During the preparation of this manuscript, the authors used GPT-5.4 as an exploratory aid in surveying the literature on limit operators and identifying connections between limit-operator theory and the problems considered here. Following a connection suggested by the model, the authors consulted Seidel's work \cite{seidel2015semi} and adapted and streamlined the proof of Theorem \ref{thm:surjectivity_implies_Fredholm}. GPT-5.4 was also used to suggest simplifications of technical parts of the arguments, including a simpler choice of the perturbation sequence $(\delta_k)_k$ in Lemma \ref{lem:good_delta_exists}, which had initially been formulated in a more complicated way. Codex 5.5 and Claude Opus 4.7 were used to assist with the Lean formalization. All mathematical arguments, formalizations and technical details were independently checked and written by the authors.

\section{Totally positive functions and Gabor frames}

This section collects the key ingredients and results about the Zak
transform of totally positive functions, a criterium for Gabor frames
via a restricted pre-Gramian matrix, and its surjectivity. These
provide the analytic preliminaries for the deeper analysis of the
pre-Gramian.

\subsection{Zak transform}

For $g \in L^2(\R)$, we define the Zak transform $Zg : [0,1)^2 \to \C$ by
$$
Zg(x,\xi)=\sum_{k\in \Z} g(x-k)e^{2\pi i k\xi}.
$$
The Zak transform of a continuous totally positive function is continuous \cite[Thm.~8.2.1.~(c)]{Groechenig2001}.
The following important theorem provides a complete description of the
zero set of the Zak transform of a  continuous totally positive
function~\cite[Thm.~1]{VinogradovUlitskaya2017} and \cite{KS14,Klo15}. 

\begin{theorem}\label{prop:zak-zero}
Let $g \in L^1(\R)$ be a continuous totally positive function. Then the Zak transform $Zg$ has a unique zero in $[0,1)^2$, located at $(x_0,1/2) \in [0,1)^2$ for some
$0 \leq x_0<1$.
\end{theorem}
Since the Zak transform is quasi-periodic and satisfies $Zg(x+k, \xi )
= e^{2\pi i k \xi } Zg(x,\xi )$, Theorem~\ref{prop:zak-zero} asserts
that \emph{all} zeros of the Zak transform of a continuous totally
positive function are of the form $(x_0 + k, 1/2), k\in \Z $. 
\subsection{Submatrix criterion}
We next recall a submatrix criterion that reduces the frame property of a Gabor system  to the invertibility of suitable pre-Gramians.
Given $g\in\lt$ and a sequence $\delta = (\delta_k)_{k\in\Z}$, we define the \emph{pre-Gramian} $\preG[\delta]$ by
\begin{equation}
        \preG[\delta] = (g(k+\delta_k-l))_{k,l\in\Z}.
    \end{equation}
We say that a function $g : \R \to \C$ has polynomial decay of order $\sigma>1$ if there exist $C > 0$ such that 
\begin{equation}\label{eq:polynomial_decay}
    |g(t)| \leq C (1+|t|)^{-\sigma}, \quad t \in \R.
\end{equation}

The following proposition shows that the Gabor frame property follows once one can select, from every lattice translate $x+\alpha\Z$, a suitably perturbed copy of $\Z$ whose associated pre-Gramian is boundedly invertible
\cite[Prop.~2.4]{Groechenig2023}. This yields a link from a discrete matrix problem back to the original Gabor system.

\begin{proposition}\label{prop:submatrix}
Let $g : \R \to \C$ be a function with polynomial decay $\sigma>1$. Further, let $0<\alpha<1$. Assume that for every
$x\in \R$, the set $x+\alpha\mathbb Z$ contains a subset
$$
\{k+\delta_k(x):k\in\mathbb Z\}
$$
with a perturbation sequence $\delta(x) = ( \delta_k(x))_{k \in \Z}$, such that the pre-Gramian $\preG[\delta][x]$ defines a bounded invertible linear operator on $\ell^\infty(\mathbb Z)$. Then
$\cG(g,\alpha,1)$ is a Gabor frame.
\end{proposition}

To apply Proposition \ref{prop:submatrix}, one must establish the invertibility of the relevant pre-Gramians. The first step comes from total positivity: if the perturbations are chosen within the interval determined by the unique zero of the Zak transform, then the associated pre-Gramian is surjective on $\ell^1(\Z)$. The following proposition makes this precise \cite[Prop.~3.1]{Groechenig2023}.

\begin{proposition}\label{prop:surjective}
Let $g \in L^1(\R)$ be a continuous totally positive function. Let
$(x_0,\frac12)\in[0,1)^2$ be the unique zero of the Zak transform
$Zg$. Let $0<\varepsilon<\frac12$, and let $\delta = (\delta_k) _{k
  \in \Z}$ be an arbitrary sequence of real numbers satisfying
$$
\delta_k\in [x_0-1+\varepsilon,\ x_0-\varepsilon],
\quad k\in \Z.
$$
Then the pre-Gramian $\preG[\delta]$ defines a surjective linear
operator on $\ell^1(\Z)$ and on $\ell ^\infty (\Z )$.
\end{proposition}

This concludes the analytic input from total positivity. For the proof
of Theorem~\ref{thm:main} it remains to show that the structured
pre-Gramians arising in our setting are not only surjective but in
fact invertible. 

\section{Operator theory tools}

The operator-theoretic argument that will give injectivity of the pre-Gramians considered in Proposition \ref{prop:surjective} relies on three ingredients. First, a spectral invariance theorem allows invertibility to be transferred between the spaces $\ell^p(\Z)$. Second, approximation by band operators shows that every surjective matrix with polynomial off-diagonal decay is Fredholm. Finally, limit-operator theory implies that a self-similar Fredholm operator is injective. Together, these results strengthen the surjectivity statement obtained in the previous section to the invertibility required by Proposition~\ref{prop:submatrix}.

In the present section, we denote by $\mathcal{L}({\ell^1(\Z)})$ the
Banach space of all bounded linear operators on ${\ell^1(\Z)}$, and
endow it with the operator norm $\| A \|_{1 \to 1}$, which can be
computed by Schur's test to be 
\begin{equation}\label{eq:op_norm_bound}
    \| A \|_{1 \to 1} 
    =\sup_{l\in\mathbb Z}\sum_{k\in\mathbb Z}|A_{kl}|.
\end{equation}

\subsection{Spectral invariance}

Let $A = (A_{kl})_{k,l \in \Z}$ be a bi-infinite matrix with entries $A_{kl} \in \C$. We say that $A$ has polynomial off-diagonal decay of order $\sigma>1$ if there exist $C>0$ such that
$$
|A_{kl}|\le C(1+|k-l|)^{-\sigma},
\quad k,l\in\mathbb Z.
$$
Every $A = (A_{kl})_{k,l \in \Z}$ with polynomial off-diagonal decay defines a bounded linear operator on $ \ell^p(\mathbb Z)$ for every $1 \leq p \leq \infty$ via the formula
$$
(Ac)_k = \sum_{l \in \Z} A_{kl}c_l.
$$
We make use of the following  result on spectral invariance
\cite{ShinSun2009,Tessera2010}, which complements the classical result of
Jaffard. 

\begin{theorem}\label{thm:spec_inv}
     Let $A=(A_{k,l})_{k,l\in\Z}$ have polynomial off-diagonal decay.
    If $A$ is invertible on $\ell^{p_0}(\Z)$ for some $1\leq p_0\leq \infty$, then it is invertible on all $\ell^p(\Z)$, $1\leq p\leq \infty$.
\end{theorem}

\subsection{Band operators}
 A bounded linear operator $A\in \mathcal L(\ell^1(\mathbb Z))$ is called a {band operator} if there exists a $w\in\mathbb N_0$ such that the matrix $(A_{kl})_{k,l \in \Z}$ corresponding to $A$ satisfies
$
A_{kl}=0
$
whenever $|k-l|>w$.
The smallest such $w$, when it exists, is called the bandwidth of $A$.

For $N\in \N$, let $P_N : \ell^1(\Z) \to \ell^1(\Z)$ denote the projection onto the coordinates $\{ -N, \dots, N \}$ and $Q_N: \ell^1(\Z) \to \ell^1(\Z)$ the projection onto the complement of $\{ -N, \dots, N \}$,

\begin{equation}\label{eq:projections}
    (P_N c)_k=
\begin{cases}
c_k,& \abs{k}\leq N\\
0,& \abs{k}>N
\end{cases},
\qquad \quad  Q_N=I-P_N.
\end{equation}

The next statement shows that the kernel of a surjective band operator cannot carry more degrees of freedom than are permitted by its bandwidth.

\begin{proposition}\label{prop:surjective_band_kernel_bound}
    Every surjective band operator $A\in\mathcal L(\ell^1(\Z))$ of
bandwidth $w\in\mathbb N_0$ satisfies
$
\dim\ker A\leq 2w .
$
\end{proposition}
\begin{proof}
Let $(A_{kl})_{k,l \in \Z}$ be the matrix corresponding to $A$. 
Since $A$ has bandwidth $w$, for every   $c \in \ell^1(\Z)$ the entry $(Ac)_k$ of is given by
$$
( Ac)_k=\sum_{|k-l|\leq w}A_{kl}c_l,
\quad k\in\mathbb Z.
$$

Let $N$ in $\N$, and consider the restriction
$$
J_N:\ell^1(\Z)\to\mathbb C^{2N+1},
\qquad
Jc=c|_{\{-N,\dots , N\}}.
$$
Further, let $\tilde A_N$ be the finite-dimensional operator
$$
\tilde A_N:\mathbb C^{2N+2w+1}\to\mathbb C^{2N+1},
\quad
(\tilde A_Nz)_k:=\sum_{|l|\leq N+w}A_{kl}z_l, \quad |k|\leq N.
$$
We claim that this map is surjective: if $y\in\mathbb C^{2N+1}$ is extended by zero to
$\tilde y\in \ell^1(\Z)$, choose $c\in \ell^1(\Z)$ with
$Ac=\tilde y$, and put $z:=J_{N+w}c$. 
If $|k|\leq N$ and
$|k-l|\le w$, then $|l|\leq N+w$. Hence, using also that
$A_{kl}=0$ whenever $|k-l|>w$, we obtain, for every $|k|\leq N$,
$$
(\tilde A_Nz)_k
=
\sum_{|l|\leq N+w}A_{kl}z_l
=
\sum_{|k-l|\le w}A_{kl}c_l
=
(Ac)_k
=
\widetilde y_k
=
y_k.
$$
Thus $\tilde A_Nz=y$.
As a surjective linear map between finite-dimensional spaces, rank-nullity gives
$$
\dim\ker \tilde A_N
=
2w .
$$

Now choose a finite-dimensional linear subspace $V\subseteq\ker A$.
Since the sequence of operators 
$(\mathrm{Id}_{\ell^1(\Z)} - P_N)_{N\in\N}$ converges pointwise to zero, the compactness of the unit ball in $V$ implies that
\begin{equation}\label{eq:QR}
    \| (\mathrm{Id}_{\ell^1(\Z)} - P_N)|_V \|_{1\to1} \to 0, \quad N \to \infty.
\end{equation}
Fix an $N$ large enough so that 
$$
\|(\mathrm{Id}_{\ell^1(\Z)}-P_N)|_V\|_{1\to1}<1.
$$
Then $P_N|_V : V \to \ell^1(\Z) $ is injective. 
Indeed, if $c \in V$ with $P_N|_V c = 0$, then $c = (\mathrm{Id}_{\ell^1(\Z)}-P_N)|_V c$. 
Using that $(\mathrm{Id}_{\ell^1(\Z)}-P_N)|_V$ is strictly contractive, we obtain $c=0$.

Now consider the restriction $J_{N+w}|_V$.
This  restriction is injective: if $J_{N+w}|_Vc=0$, then $P_N c =0$, and since $c \in V$, we have $c=0$ because $P_N|_V$ is injective (shown before). 
Moreover, if  $c\in V$, then $Ac=0$ (because $V$ is a subspace of the kernel of $A$) and so for every $|k|\leq N$ we have
$$
0 = (Ac)_k = \sum_{|l|\leq N+w} A_{kl} c_l = (\tilde A_N J_{N+w} c)_k.
$$
Hence, $\tilde A_N J_{N+w}c=0$ for all $c\in V$. Thus, $J_{N+w}$ embeds $V$
into $\ker \tilde{A}_N$,  and this fact  implies that
$$
\dim V\leq\dim\ker \tilde{A}_N=2w.
$$
Since $V$ was an arbitrary finite-dimensional subspace of $\ker A$,  we conclude that $\ker A$ is finite-dimensional, with $\dim \ker A\leq 2w$.
\end{proof}

\subsection{Surjectivity implies Fredholm}

Recall that a linear operator $A$ is called \emph{Fredholm} if $\operatorname{ran}A$ is closed and both $\dim\ker A$ and $\dim\operatorname{coker}A$ are finite-dimensional.
The purpose of this section is to prove that every surjective operator
on $\ell^1(\Z)$ whose matrix has polynomial off-diagonal decay is
Fredholm. The proof of this statement can be deduced from Seidel's
theory on semi-Fredholm band-dominated operators (limits of band
operators): since $A$ is surjective, it has closed range and has a
finite-dimensional cokernel. This implies that $A$ belongs to the
class of so-called lower semi-Fredholm operators. Since $A$ has
polynomial off-diagonal decay, it is a band-dominated operator and
therefore \cite[Thm.~4.3]{seidel2015semi} implies that $A$ is
Fredholm. We now present an elementary and self-contained proof  of
this statement that avoids the notions of $\mathcal{P}$-compactness
and Bernstein numbers used in Seidel's work \cite{seidel2015semi}.  

Observe that Proposition \ref{prop:surjective_band_kernel_bound} states that a surjective band operator $A\in \mathcal{L}(\ell^1(\Z))$ is a Fredholm operator.
\begin{theorem}\label{thm:surjectivity_implies_Fredholm}
If $A = (A_{kl})_{k,l \in \Z}$ possesses  polynomial off-diagonal
decay $\sigma >1$ and  is surjective on $\ell^1(\Z)$, then $A$ is Fredholm on $\ell^1(\Z)$. In particular, $A$ has a finite-dimensional kernel. 
\end{theorem}

\begin{proof}
It suffices to show that $A$ has a finite-dimensional kernel.

\textbf{Step 1. A perturbation argument.}
Let $\tilde{A},R\in \mathcal{L}(\ell^1(\Z))$ with
\begin{equation} \label{eq:c1}
    AR=\mathrm{Id}_{\ell^1(\Z)},\quad \|A-\tilde{A}\|_{1\to 1} \norm{R}_{1\to 1} <1.
\end{equation}
We claim that there exists a bounded invertible operator $\Phi$ with $A=\tilde{A}\Phi$. In particular, $\tilde{A}$ is surjective and
\begin{equation}\label{eq:same_kernel}
    \dim \ker A= \dim\ker \tilde{A} .
\end{equation}

Set $E=A-\tilde{A}$. By assumption~\eqref{eq:c1}, $\|ER\|_{1\to 1}<1$. 
Thus, the operator $\mathrm{Id}_{\ell^1(\Z)}- ER$ is invertible. Hence, from 
\begin{equation}
    \tilde{A}R = (A- E)R =\mathrm{Id}_{\ell^1(\Z)} -  ER,
\end{equation}
it follows that $\tilde{A}$ is surjective.
We set 
\begin{align}
    \Phi & = \mathrm{Id}_{\ell^1(\Z)} +R\sum\limits_{l=0}^\infty (ER)^l E
     = \mathrm{Id}_{\ell^1(\Z)} +\sum\limits_{l=1}^\infty (RE)^l = (\mathrm{Id}_{\ell^1(\Z)} -RE)^{-1}.
\end{align}
Clearly, $\Phi$ is invertible with inverse $\mathrm{Id}_{\ell^1(\Z)}
-RE$, and thus 
\begin{equation}
    A\Phi^{-1} = A(\mathrm{Id}_{\ell^1(\Z)} -RE) = A- ARE = A-E = \tilde{A}.
\end{equation}
\textbf{Step 2. Existence of a right inverse. }
Since $A$ is surjective, the open mapping theorem implies that there exists a constant
$C>0$ such that for every $c\in {\ell^1(\Z)}$ there exists $b\in {\ell^1(\Z)}$ with
$$
Ab=c,
\quad
\|b\|_1\leq C\|c\|_1 .
$$
For each $l\in\mathbb Z$, choose $b^{(l)}\in {\ell^1(\Z)}$ such that
$$
Ab^{(l)}=e_l,
\quad
\|b^{(l)}\|_1\leq C,
$$
and define
$$
Rc:=\sum_{l\in\mathbb Z}c_l b^{(l)} .
$$
This series converges absolutely in ${\ell^1(\Z)}$, and
$$
\|Rc\|_1\leq C\|c\|_1 .
$$
Thus, $R\in\mathcal L({\ell^1(\Z)})$. Since the defining series is absolutely
convergent, applying $A$ termwise gives
$$
ARc=\sum_{l\in\mathbb Z}c_lAb^{(l)}
=\sum_{l\in\mathbb Z}c_le_l
=c.
$$
Therefore, we have $AR=\mathrm{Id}_{\ell^1(\Z)}$.

\textbf{Step 3. Approximation by band operators.}
Let $A_n$ be the band truncation of $A$  defined by
$$
(A_n)_{kl}=
\begin{cases}
A_{kl}, & |k-l|\leq n,\\
0, & |k-l|>n .
\end{cases}
$$
Then 
$
\|A-A_n\|_{1\to1} \to 0
$
as $n \to \infty$. Indeed, 
   let $C>0$, $\sigma>1$ be such that 
   \begin{equation}
       |A_{kl}|\le C(1+|k-l|)^{-\sigma},
\quad k,l\in\mathbb Z.
   \end{equation}
   Then by \eqref{eq:op_norm_bound}, we have
   \begin{align}
       \norm{A-A_n}_{1 \to 1} & =\sup\limits_{l\in\Z}\sum\limits_{k\in\Z}|(A-A_n)_{kl}| = \sup\limits_{l\in\Z}\sum\limits_{|k-l|>n}|A_{kl}| \\
       & \leq  C \sup\limits_{l\in\Z}\sum\limits_{|k-l|>n} (1+|k-l|)^{-\sigma} 
       = C \sum\limits_{|j|>n} (1+|j|)^{-\sigma}.
   \end{align}
   Since $\sum\limits_{j\in\Z} (1+|j|)^{-\sigma}$ converges, $\norm{A-A_n}_{1\to 1}\to 0$ as $n\to \infty$.

\textbf{Step 4.}
Fix a large enough $n$ so that
$$
\|A-A_n\|_{1\to1}\,\|R\|_{1\to1}<1 .
$$

By Step 1, $A_n$ is surjective. 
Applying \eqref{eq:same_kernel} and Proposition~\ref{prop:surjective_band_kernel_bound} to the surjective band operator $A_n$, we obtain
\begin{equation}
    \dim\ker A= \dim\ker A_n \leq 2n.
\end{equation}
\end{proof}
\begin{remark} {\rm
  The above proof is streamlined for $\ell ^1(\Z )$, but it contains 
  several more general facets. For instance, the perturbation argument
  works on arbitrary Banach spaces and, in particular, implies the
  well-known statement that the Fredholm index is locally
  constant. Step~2 constructs a bounded right-inverse for operators on
  $\ell ^1(\Z )$; such a bounded inverse exists always on $\ell ^2(\Z
  )$ via the pseudo-inverse. On general Banach spaces, one must assume
  that the projection onto the kernel of $A$ is bounded. With that
  additional assumption, 
  Theorem~\ref{thm:surjectivity_implies_Fredholm} holds also for Banach spaces.  }
\end{remark}

\subsection{Limit operators and injectivity}

Let the projections $P_N$ and $Q_N$ be defined as in \eqref{eq:projections}. 
To that, we define for $m \in \Z$ the shift operator $T_m : \ell^1(\Z) \to \ell^1(\Z)$ by
$$
(T_m c)_k=c_{k-m}.
$$
We now introduce the notions of $\mathcal{P}$-topology and operator spectrum, see \cite{lindner2006infinite,rabinovitch2004limit}.

\begin{definition}
A sequence $(A_n)_{n \in \N} \subset \mathcal{L}(\ell^1(\Z))$ of bounded linear operators on $\ell^1(\Z)$ is said to
converge to $A \in \mathcal{L}(\ell^1(\Z))$ in the \emph{$\mathcal P$-topology} if for every
$N\in \N$ 
$$
\lim_{n \to \infty} \left(\norm{P_N(A_n-A)}_{1\to1}+\norm{(A_n-A)P_N}_{1\to1}\right) = 0.
$$
An operator $B \in \mathcal{L}(\ell^1(\Z))$ is called a \emph{limit operator} of $A$, if there exists a sequence $(m_n)_{n \in \N}\subset \Z$ with $\abs{m_n}\to \infty$ so that
$$
T_{-m_n}AT_{m_n}\to B
$$
in the $\mathcal P$-topology. The set of all limit operators of $A$ is called the
\emph{operator spectrum} of $A$ and is denoted $\sigma_{\mathrm{op}}(A)$.
An operator $A \in \mathcal{L}(\ell^1(\Z))$ is called \emph{self similar} if $A\in \sigma_{\mathrm{op}}(A).$
\end{definition}

The next theorem concerns limit operators of Fredholm operators on $\ell^1(\Z)$. To place it in context, recall first the stronger result available in the band-dominated setting. Every band-dominated operator on $\ell^1(\Z)$ is rich, and, in this setting, $\mathcal P$-Fredholmness is equivalent to ordinary Fredholmness. Consequently, a band-dominated operator $A$ on $\ell^1(\Z)$ is Fredholm if and only if every $B\in\sigma_{\mathrm{op}}(A)$ is invertible, see \cite[Thm.~5.7]{seidel2015semi} and \cite[Thm.~11]{lindner2014affirmative}. In the theorem below, $A\in\mathcal L(\ell^1(\Z))$ is an arbitrary bounded operator and is not assumed to be band-dominated. Although the cited results do not apply in this generality, we show that Fredholmness of $A$ still implies that every limit operator $B\in\sigma_{\mathrm{op}}(A)$ is injective.


\begin{theorem}\label{thm:selfsimilarinjective}
If $A \in \mathcal L(\ell^1(\Z))$ is Fredholm, then every
$B \in \sigma_{\mathrm{op}}(A)$ is injective.
\end{theorem}

\begin{proof}
Let $B\in\sigma_{\mathrm{op}}(A)$. By definition of the operator spectrum,
there exists a sequence $(m_n)_{n\in\mathbb N}\subset\Z$ with
$\abs{m_n}\to\infty$ such that
$$
A_n:=T_{-m_n}AT_{m_n}\to B
$$
in the $\mathcal P$-topology. Hence, for every $N\in\N$,
\begin{equation}\label{eq:p-topology-column}
\lim_{n\to\infty}\norm{(A_n-B)P_N}_{1\to1}=0.
\end{equation}
Set $V:=\ker A$. Since $A$ is Fredholm, $V$ is finite-dimensional and
$\operatorname{ran}A$ is closed.

\textbf{Step 1.}
We first prove that for every finitely supported non-zero sequence $u\in \ell^1(\Z)$,
\begin{equation}\label{eq:shifted-distance-limit}
\dist(T_{m_n}u,V)\to \norm{u}_1,
\end{equation}
where $\dist(c,V):=\inf_{v\in V}\norm{c-v}_1$.

Fix $0<\varepsilon<1/2$. Since $V$ is finite-dimensional, the compactness of the unit ball in $V$
yields an $M\in\N$ such that
\begin{equation}\label{eq:kernel-tail-small}
\norm{Q_M v}_1\le \varepsilon\norm{v}_1,
\quad \text{ for all }  v\in V,
\end{equation}
cf. \eqref{eq:QR}.
Let $S_n:=\operatorname{supp}(T_{m_n}u)$. Since $u$ has finite support and
$\abs{m_n}\to\infty$, there exists an $n_0$ such that for all $n>n_0$
\begin{equation}\label{eq:shifted-support-away}
S_n\subset \Z\setminus\{-M,\ldots,M\}.
\end{equation}
Fix $n>n_0$ and let $v\in V$. If $\norm{v}_1>2\norm{u}_1$, then
\begin{equation}\label{eq:large-kernel-vector-distance}
\norm{T_{m_n}u-v}_1
\ge
\norm{v}_1-\norm{T_{m_n}u}_1
=
\norm{v}_1-\norm{u}_1
>
\norm{u}_1
\ge
(1-2\varepsilon)\norm{u}_1.
\end{equation}
On the other hand, if $\norm{v}_1\le 2\norm{u}_1$, then
$$
\norm{Q_M v}_1
\le
\varepsilon\norm{v}_1
\le
2\varepsilon\norm{u}_1.
$$
Since $T_{m_n}u$ is supported on
$S_n\subset\Z\setminus\{-M,\ldots,M\}$, we have
$Q_MT_{m_n}u=T_{m_n}u$. Hence
\begin{equation}\label{eq:small-kernel-vector-distance}
\begin{aligned}
\norm{T_{m_n}u-v}_1
&\ge
\norm{Q_M(T_{m_n}u-v)}_1 \\
&=
\norm{T_{m_n}u-Q_Mv}_1 \\
&\ge
\norm{T_{m_n}u}_1-\norm{Q_Mv}_1 \\
&\ge
(1-2\varepsilon)\norm{u}_1 .
\end{aligned}
\end{equation}
Combining the two cases, we obtain, for all $n>n_0$,
$$
\norm{T_{m_n}u-v}_1
\ge
(1-2\varepsilon)\norm{u}_1,
\quad v\in V.
$$
Taking the infimum over $v\in V$ gives
$$
\dist(T_{m_n}u,V)
\ge
(1-2\varepsilon)\norm{u}_1.
$$
Therefore
$$
\liminf_{n\to\infty}\dist(T_{m_n}u,V)
\ge
(1-2\varepsilon)\norm{u}_1.
$$
Letting $\varepsilon\downarrow0$, we obtain
$$
\liminf_{n\to\infty}\dist(T_{m_n}u,V)\ge \norm{u}_1.
$$
The reverse inequality always holds,  because $0\in V$, and hence
$$
\dist(T_{m_n}u,V)
\le
\norm{T_{m_n}u}_1
=
\norm{u}_1.
$$
This proves \eqref{eq:shifted-distance-limit}.

\textbf{Step 2.} Suppose, towards a contradiction, that $B$ is not injective.
By \cite[Thm.~5.2, p.~231]{kato1980perturbation},  
there exists a $C>0$ such that
\begin{equation}\label{eq:fredholm-lower-estimate}
\norm{Ac}_1\ge C\,\dist(c,V),
\quad \text{ for all } c\in\ell^1(\Z).
\end{equation}
Since $B$ is not injective, there exists
$0\ne c\in\ell^1(\Z)$ such that $Bc=0$. Since $P_Nc\to c$ and
$Q_Nc\to0$ in $\ell^1(\Z)$, and since
$\norm{P_Nc}_1\to\norm{c}_1>0$, we may choose $N\in\N$ such that
$$
u:=P_Nc\ne0
$$
and
\begin{equation}\label{eq:tail-choice}
\norm{B}_{1\to1}\,\norm{Q_Nc}_1
<
\frac C8\norm{u}_1.
\end{equation}
Using $Bc=0$, we get
$$
Bu=B(P_Nc)=B(P_Nc-c)=-BQ_Nc,
$$
and therefore
\begin{equation}\label{eq:Bu-small}
\norm{Bu}_1
\le
\norm{B}_{1\to1}\,\norm{Q_Nc}_1
<
\frac C8\norm{u}_1.
\end{equation}
Since $u=P_Nu$, equation \eqref{eq:p-topology-column} implies that
$$
\norm{(A_n-B)u}_1
\le
\norm{(A_n-B)P_N}\,\norm{u}_1
\to 0.
$$
Hence, for all sufficiently large $n$, we have
\begin{equation}\label{eq:Aj-minus-B-small}
\norm{(A_n-B)u}_1
<
\frac C8\norm{u}_1.
\end{equation}
Also, by \eqref{eq:shifted-distance-limit},
\begin{equation}\label{eq:distance-large}
\dist(T_{m_n}u,V)
>
\frac34\norm{u}_1,
\end{equation}
for all sufficiently large $n$. Choose $n$ large enough so  that both \eqref{eq:Aj-minus-B-small} and
\eqref{eq:distance-large} hold.
Since 
$T_{m_n}$ is an isometry, we obtain
\begin{equation}\label{eq:Av-upper}
\begin{aligned}
\norm{AT_{m_n}u}_1
&=
\norm{T_{-m_n}AT_{m_n}u}_1 \\
&=
\norm{A_nu}_1 \\
&\le
\norm{(A_n-B)u}_1+\norm{Bu}_1 \\
&<
\frac C4\norm{u}_1.
\end{aligned}
\end{equation}
On the other hand, applying \eqref{eq:fredholm-lower-estimate} to
$z=T_{m_n}u$, and then using \eqref{eq:distance-large}, gives
\begin{equation}\label{eq:Av-lower}
\norm{AT_{m_n}u}_1
\ge
C\,\dist(T_{m_n}u,V)
>
\frac{3C}{4}\norm{u}_1.
\end{equation}
The estimates \eqref{eq:Av-upper} and \eqref{eq:Av-lower} contradict each
other. Hence, $B$ is injective.
\end{proof}

\section{Proof of Theorem \ref{thm:main}}

We now combine the results of the previous sections to prove the main theorem.
By the operator-theoretic criterion formulated in
Proposition~\ref{prop:submatrix} above, it remains to construct, 
for each coset $x+\alpha\mathbb Z$, a perturbation sequence satisfying two
properties. First, the associated pre-Gramian must be self-similar, so that the
injectivity result from limit-operator theory applies. Second, the
perturbations must stay within the interval determined by the unique
zero of the Zak-transform, so that  surjectivity of the pre-Gramian
can be  ensured by
Proposition~\ref{prop:surjective}. The next two results establish these
properties.

We begin with a technical lemma that details the continuous dependence
of the pre-Gramian $G_\delta $ on $\delta $. 
\begin{lemma}\label{lem:uniformcontinuity}
Let $g:\mathbb R\to\mathbb C$ be a continuous function with polynomial
decay of order $\sigma>1$, and  let $R>0$. Then the following holds.

\begin{enumerate}[(i)]
\item Assume that  $(u_n)_{n\in\mathbb N}\subset \R$  and $
  \lim_{n\to\infty}u_n=u, $
then
\begin{equation}
  \label{eq:c6}
\lim_{n\to\infty}
\sup_{x\in\mathbb R}|g(x+u_n)-g(x+u)|=0.
\end{equation}
\item Pointwise convergence:   Let $(u_n(k))_{n \in \N} \subset [-R,R]$ and $u(k) \in [-R,R]$, $k \in \Z$. Assume that for every $k\in\mathbb Z$, we have
$
\lim_{n\to\infty}u_n(k)=u(k).
$
Then for every $l\in\mathbb Z$,
\begin{equation}
  \label{eq:c9}
\lim_{n\to\infty}
\sum_{k\in\mathbb Z}
|g(k-l+u_n(k))-g(k-l+u(k))|
=0.
\end{equation}
\item Uniform convergence:  Let $ u_n = (u_n(k))_{k \in \Z} \subset
  [-R,R]$ and $u = (u(k))_{k\in \Z}  \subset  [-R,R]$, $k \in
  \Z$. Assume that 
$
\lim_{n\to\infty} \sup _{k\in \Z } |u_n(k)-u(k)| = 0.
$
Then the corresponding pre-Gramian matrices $G_{u_n}$ converge in norm
to $G_u$:
$$
\lim _{n\to \infty } \|G_{u_n} - G_u\|_{1 \to 1}  = 0 \, .
$$
\end{enumerate}
\end{lemma}
\begin{proof}
(i) Since $g$ is continuous and vanishes at infinity, $g$ is in fact
uniformly continuous. The claim  \eqref{eq:c6} follows from the
uniform continuity. 

(ii) Let $b_k := \sup _{|x|\leq R} (1+|k+x|)^{-\sigma } \leq C_R
(1+|k|)^{-\sigma }$. Since $\sigma >1$, the sequence $b= (b_k)_{k\in
  \Z }$ is summable.

Then for all $n\in\mathbb N$ and $k\in\mathbb Z$, we have 
$$
|g(k-l+u_n(k))-g(k-l+u(k))|
\leq
|g(k-l+u_n(k))|+|g(k-l+u(k))|
\leq 2b_{k-l}.
$$
Consequently, given $\varepsilon >0$, we can find an $N\in \N$, such that
\begin{align} \label{eq:c7}
  \sum_{k : |k-l| >N }
|g(k-l+u_n(k))-g(k-l+u(k))| \leq 2 \sum _{k : |k-l| >N } b_{k-l} <
  \varepsilon \, 
\end{align}
independent of $l$. 
 Fix $k,l\in\mathbb Z$. Since $u_n(k)\to u(k)$ as $n \to \infty$ and
 since $g$ is continuous, we obtain 
$$
|g(k-l+u_n(k))-g(k-l+u(k))| \to 0,
\quad n\to\infty.
$$
The sum in \eqref{eq:c9}  consists of terms that converge to $0$ and is dominated by
the summable sequence $(b_k)_{k\in\Z}$. Thus an application of the dominated convergence theorem yields the claim.

(iii) We argue similar to (ii). Given $\varepsilon >0$, we find $N  \in
\N $, such that  for all $l\in \Z $ 
$$
\sum_{k : |k-l| >N }
|g(k-l+u_n(k))-g(k-l+u(k))|  < \varepsilon /2\, .
$$
Since $g$ is uniformly continuous,  there is $\eta >0$ such that
$|u-v|<\eta$ implies 
$\sup _{x\in \R } |g(x+u) - g(x+v)| < \frac{\varepsilon}{
  2(2N+1)}$. Using the uniform convergence of $u_n $ to $ u$, we find
$n_0 \in \N$ such that $\sup _{k\in\Z } |u_n(k) - u(k)| < \eta $ for
$n\geq n_0$. Then 
$$
\sum_{k : |k-l| \leq N }
|g(k-l+u_n(k))-g(k-l+u(k))|  \leq (2N+1) \sup _{x\in \R }
|g(x+u_n(k))-g(x+u( k))| <  \varepsilon /2\, , 
$$
which is again independent of $l$.

Adding the contributions for $|k-l|\leq N$ and $|k-l|>N$, we obtain
for all $n\geq n_0$
$$
\|G_{u_n } - G_u\|_{1\to 1} = \sup _{l\in \Z } \sum _{k\in \Z }
|g(k-l+u_n(k))-g(k-l+u(k))| < \varepsilon \, .
$$
\end{proof}

Lemma \ref{lem:uniformcontinuity} provides the key step for establishing self-similarity. It shows that pointwise convergence of bounded perturbation sequences implies convergence of the corresponding pre-Gramians in the
$\mathcal P$-topology. Consequently, every recurrent perturbation sequence gives rise to a self-similar pre-Gramian.

\begin{proposition}
\label{thm:preG_selfsim}
    Let $g:\R\to\C$ be a continuous function with polynomial decay of order $\sigma>1$, and $\delta=(\delta_k)_{k\in\Z}$ a bounded sequence.
    If there exists a sequence  $(m_n)_{n\in\Z} \subseteq \Z$ with $|m_n|\to\infty$ as $n\to\infty$ such that 
    \begin{equation} \label{eq:c2} 
        \lim\limits_{n\to\infty} \delta_{k+m_n}=\delta_k,\quad k\in\Z,
    \end{equation}
    then the pre-Gramian
    $\preG[\delta]=(g(k+\delta_k-l))_{k,l\in\Z}$ is self-similar. 
\end{proposition}
Consistent with the terminology of ergodic theory, we call a sequence
satisfying \eqref{eq:c2} a recurrent sequence. 
\begin{proof}
We set $G = \preG[\delta]$.
  Choose $R>0$ such that $ (\delta_k)_{k\in\Z}  \subseteq [-R,R]$.
Set
$$
M_n:=T_{-m_n}GT_{m_n}-G .
$$
We prove that, for every fixed $N\in\mathbb N$,
$$
\lim_{n\to\infty}\|M_nP_N\|_{1\to1}=0
\qquad\text{and}\qquad
\lim_{n\to\infty}\|P_NM_n\|_{1\to1}=0.
$$
For $m\in\mathbb Z$, a direct calculation shows that the matrix entries of $T_{-m}GT_m$ are given by
$$
(T_{-m}GT_m)_{kl}=G_{k+m,l+m},\quad k,l\in\Z.
$$
Therefore, the entries of $M_n$ are
$$
(M_n)_{kl}
=
G_{k+m_n,l+m_n}-G_{kl}
=
g(k+\delta_{k+m_n}-l)-g(k+\delta_k-l).
$$
We first estimate the operator norm of 
 $M_nP_N$. 
Since the only non-zero columns of $M_nP_N$ are in the range
$|l|\leq N$, 
an application of Schur's test~\eqref{eq:op_norm_bound}
gives
\begin{equation}\label{eq:column}
    \|M_nP_N\|_{1\to1}
=
\max_{|l|\leq N}
\sum_{k\in\mathbb Z}
\left|
g(k+\delta_{k+m_n}-l)-g(k+\delta_k-l)
\right|.
\end{equation}
Fix $l\in\mathbb Z$. Since $\delta_k\in  [-R,R]$ for all
$k$, and
$
\delta_{k+m_n}\to\delta_k$ as $n\to\infty$
for every fixed $k\in\mathbb Z$, Lemma~\ref{lem:uniformcontinuity}
\textup{(ii)}, applied with $u_n(k)=\delta_{k+m_n}$  and
$u(k)=\delta_k$, gives
$$
\lim_{n\to\infty}
\sum_{k\in\mathbb Z}
\left|
g(k-l+\delta_{k+m_n})-g(k-l+\delta_k)
\right|
=0.
$$
Since the maximum in \eqref{eq:column} is taken over the fixed finite set $\{-N,\dots, N\}$, we obtain
$$
\lim_{n\to\infty}\|M_nP_N\|_{1\to1}=0.
$$

We now estimate
$P_NM_n$. 
Since the only non-zero rows of $P_NM_n$ are in the range
$|k|\leq N$, 
an application of 
\eqref{eq:op_norm_bound}
implies that
\begin{equation}\label{eq:PNMn}
\begin{split}
    \|P_N M_n\|_{1\to1}
&=
\sup_{l\in\mathbb Z}
\sum_{|k|\leq N}
\left|
g(k+\delta_{k+m_n}-l)-g(k+\delta_k-l)
\right| \\
&\leq \sum _{|k|\leq N}  \sup_{x\in\mathbb R}
\left|
g(x+\delta_{k+m_n})-g(x+\delta_k)
\right|. 
\end{split}
\end{equation}
Since $N$ is fixed
and $
\delta_{k+m_n}\to\delta_k
$
as $n \to \infty$, 
Lemma~\ref{lem:uniformcontinuity}
\textup{(i)}  implies that each summand 
on the right-hand side of \eqref{eq:PNMn} tends
to zero. The sum is over the fixed finite set $\{-N, \dots, N\}$, hence
$$
\lim_{n\to\infty}\|P_N M_n\|_{1\to1}=0.
$$

Thus, for every fixed $N\in\mathbb N$,
$$
\lim_{n\to\infty}
\bigl(\|P_N M_n\|_{1\to1}+\|M_nP_N\|_{1\to1}\bigr)=0.
$$
This shows that $G=\preG[\delta]$ is self-similar.
\end{proof}

We now combine the operator-theoretic results into an abstract criterion for the Gabor frame property. Its hypotheses separate the two remaining tasks:
establishing surjectivity of the pre-Gramian and constructing perturbation sequences with the required recurrence property.

\begin{theorem}\label{thm:semimain}
    Let $g:\R\to\C$ be a continuous function with polynomial decay of order $\sigma>1$, and $0<\alpha<1$.
    Assume that for all 
    $x\in\R$ the set $x+\alpha\Z$ contains a subset
    $\{k+\delta_k(x):k\in\Z\}\subseteq x+\alpha\Z$  with a bounded
    sequence $\delta (x) \in \ell ^\infty (\Z )$ such that the pre-Gramian $\preG[\delta][x]$ is surjective on $\ell^1(\Z)$, and that there exists a sequence $(m_n)_{n\in\Z} \subseteq \Z$ with $|m_n|\to\infty$ as $n\to\infty$ such that 
    \begin{equation}\label{eq:thm:semimain_inj}
        \lim\limits_{n\to\infty} \delta_{k+m_n}(x) = \delta_k(x),\quad k\in\Z.
    \end{equation}
    Then $\cG (g,\alpha,1)$ is a frame.
\end{theorem}
\begin{proof}
Let $x\in\R$.
By Theorem \ref{thm:surjectivity_implies_Fredholm}, $\preG[\delta][x]$ is Fredholm.
By Proposition \ref{thm:preG_selfsim}, $G_\delta $  is also self-similar. 
Theorem \ref{thm:selfsimilarinjective} further implies that $\preG[\delta][x]$ is  injective on $\ell^1(\Z)$, hence invertible on $\ell^1(\Z)$. 
Now Theorem \ref{thm:spec_inv} \textup{(i)} implies that it is invertible on $\ell^\infty(\Z)$. 
Finally, Proposition \ref{prop:submatrix} implies that $\cG(g,\alpha,1)$ is a Gabor frame.
\end{proof}

Theorem \ref{thm:semimain} reduces the proof of the main theorem to a
selection problem. For each coset $x+\alpha\mathbb Z$, we must choose
perturbations that lie in the interval of Proposition
\ref{prop:surjective}, i.e., avoid the zero of the Zak transform,
and, in addition,  satisfy the required recurrence
property~\eqref{eq:c2}.  The following lemma provides such an explicit
choice. 

\begin{lemma}\label{lem:good_delta_exists}
    Let $0<\alpha<1$, $x \in \R$, and $0<\varepsilon<\frac{1-\alpha}{2}$. Let $I \subset \R $ be an interval of length $1-2\varepsilon$. Then there exists an $a \in \R $ such that the bounded sequence $(\delta_k(x))_{k \in \Z}$ defined by
    $$
    \delta_k(x) = a+\alpha F\left(\tfrac{k+a-x}{\alpha}\right), \quad F(t)=\lceil t\rceil-t,
    $$
    satisfies
    \begin{enumerate}[(i)]
        \item $k+\delta_k(x)\in x+\alpha\Z$ for all $k\in\Z$,\item $\delta_k \in I$ for all $k \in \Z$,
        \item there exists a sequence  $(m_n)_{n\in\Z} \subseteq \Z$ with $|m_n|\to\infty$ as $n\to\infty$ such that 
        \begin{equation}
        \lim\limits_{n\to\infty} \delta_{k+m_n}(x)=\delta_k(x),\quad k\in\Z.
    \end{equation}
    \end{enumerate}
\end{lemma}
\begin{proof}
    The interval $I$ has length $1-2\varepsilon>\alpha$. Combining
    this with the fact that the set  $x+\alpha\Z+\Z$ is countable, it follows that there exists $a \in \R$ such that $[a,a+\alpha)\subset I$ and $a\notin x+\alpha\Z+\Z$. Fix an $a$ with these properties.

       For all $k\in\Z$,
    \begin{equation}
        k+\delta_k(x) = k+a+\alpha \left\lceil \tfrac{k+a-x}{\alpha}\right\rceil - (k+a-x) = x+ \alpha \left\lceil \tfrac{k+a-x}{\alpha}\right\rceil\in x+\alpha\Z.
    \end{equation}
    
Since
$
a\notin x+\alpha\mathbb Z+\mathbb Z,
$
we have $\frac{k+a-x}{\alpha}\notin\mathbb Z$. 
Therefore, $F(\frac{k+a-x}{\alpha})\in(0,1)$.
Hence, 
$
\delta_k(x)=a+\alpha F(\frac{k+a-x}{\alpha})\in(a,a+\alpha),
$
which implies that
$
\delta_k(x)\in (a,a+\alpha) \subseteq I.
$

Finally, we show Property (iii). We distinguish between $\alpha\notin\Q$ and $\alpha\in\Q$.
We denote with $\{ t\} = t-\lfloor t\rfloor$ the fractional part of $t\in\R$.

\textbf{Case 1: irrational density.} Assume that $0<\alpha<1$ is irrational. It follows that also $\frac1\alpha$ is irrational and hence
$$
\left\{\left\{ \tfrac{m}{\alpha} \right\} 
:m\in\Z\right\}
$$
is dense in $[0,1)$. Therefore, we may choose a strictly increasing
sequence $(m_n)_{n \in \N} \subset\mathbb N$ such that
\begin{equation}
   |\left\{\tfrac{m_n}{\alpha} \right\}|<\tfrac{1}{n}
   ,\quad n\in\N.
\end{equation}

Since $\frac{k+a-x}{\alpha}\notin\mathbb Z
$ for every $k \in \Z$ and since $F$ is $1$-periodic and continuous on the open set $\R\setminus\Z$, it follows that $F$ is continuous at each point $\frac{k+a-x}{\alpha}$. Hence, for all
$k\in\mathbb Z$ we have
\begin{align}
\delta_{k+m_n}
&=
a+\alpha F\left(\tfrac{k+a-x}{\alpha}+\tfrac{m_n}{\alpha}\right)\\
&=
a+\alpha F\left(\tfrac{k+a-x}{\alpha}+\left\{\tfrac{m_n}{\alpha} \right\}\right).
\end{align}
Since $|\left\{\tfrac{m_n}{\alpha} \right\}| \to 0$ as $n \to \infty$, it follows that for every $k \in \Z$ we have
$$
\lim_{n \to \infty} \delta_{k+m_n} = a+\alpha F\left(\tfrac{k+a-x}{\alpha}\right)
=
\delta_k .
$$

\textbf{Case 2: rational density.} 
Assume that $\alpha = \tfrac{p}{q}$ for some $p,q\in\N$. 
In  this case, we choose 
    $m_n = pn$, $n\in\N$. 
Then for all $k\in\Z$ we have 
\begin{equation}
    \delta_{k+m_n}
=
a+\alpha F\left(\tfrac{k+a-x}{\alpha}+qn\right) =
a+\alpha F\left(\tfrac{k+a-x}{\alpha}\right) = \delta_{k}.
\end{equation}
This proves the claim.
\end{proof}

Note that for rational density, the perturbation sequence constructed in
\cite{Groechenig2023} is periodic, as in Case~2 above.

We now complete the proof of Theorem~\ref{thm:main}. After reducing the lattice
to the form $(\alpha,1)$ by a dilation, Proposition~\ref{prop:surjective}
provides surjectivity of the pre-Gramian, while
Lemma~\ref{lem:good_delta_exists} supplies a perturbation sequence satisfying
the recurrence hypothesis of Theorem~\ref{thm:semimain}.

\begin{proof}[Proof of Theorem \ref{thm:main}]
We start by observing that $\cG(g,\alpha,\beta)$ is a frame if and only if $\cG(\tilde{g},\alpha \beta,1)$  is a frame, where $\tilde{g}(t)=g(\tfrac{t}{\beta})$ \cite[Prop.~9.4.4., or p.~200]{Groechenig2001}. 
Since the class of continuous and integrable totally positive functions is closed under dilations, it suffices to prove the statement that $\cG(g,\alpha,1)$ is a frame for all  $0<\alpha<1$.
To prove this, let $0<\varepsilon<\tfrac{1-\alpha}{2}$.
Let $(x_0,1/2)$ denote the unique zero of the Zak transform of $g$, cf. Proposition \ref{prop:zak-zero},
and set
\begin{equation}
    I = [x_0-1+\varepsilon, x_0-\varepsilon].
\end{equation}
For each $x\in\R$, select $\delta(x)$ as in Lemma \ref{lem:good_delta_exists}. With this selection, we verify that Theorem \ref{thm:semimain} is applicable.
Firstly, by Lemma \ref{lem:good_delta_exists} \textup{(i)},
\begin{equation}
    \{k+\delta_k(x):k\in\Z\}\subseteq x+\alpha\Z,\quad x\in\R.
\end{equation}
Secondly, by Proposition \ref{prop:surjective} and Lemma \ref{lem:good_delta_exists} \textup{(ii)}, $\preG[\delta][x]$ is surjective. 
By Lemma \ref{lem:good_delta_exists} \textup{(iii)}, the condition \eqref{eq:thm:semimain_inj} of Theorem \ref{thm:semimain} is satisfied. Therefore, all conditions in Theorem \ref{thm:semimain} are fulfilled, and $\cG(g,\alpha,1)$ is a frame.
\end{proof}

\section{Kadets-type theorems for shift-invariant spaces}
\subsection{Shift-invariant spaces and Gabor frames} \label{sec:52}
For the proof of the frame set conjecture (Theorem~\ref{thm:main}) we
constructed a special subset $\{k+\delta _k : k\in \Z \}$ of $x+\alpha
\Z $ and showed that the associated pre-Gramian matrix $G_\delta $
with entries $(G_\delta )_{kl} = g(k+\delta _k -l)$ is invertible. In
the context of shift-invariant spaces it is meaningful to understand
for  which perturbation sequences $(\delta _k)_{k\in \Z }$ the
pre-Gramian $G_\delta $ is invertible. The answer to this question
leads to a Kadets-type theorem analogous to the results for
bandlimited functions.

The  shift-invariant space $V^p(g)$, $1\leq p \leq \infty$,  generated by a function $g\in
L^p(\R)$ is defined by 
\begin{equation}
    V^p(g) = \Big\{ \sum\limits_{l\in\Z} c_l T_l g \in L^p(\R): c\in \ell^p(\Z)\Big\},
\end{equation}
endowed with the $L^p$-norm, $1\leq p\leq \infty$. We assume that the
$L^p$-norm on $V^p(g)$ is equivalent to the $\ell ^p$-norm of its
coefficients, i.e., $ A \|f\|_p \leq \|c\|_p \leq B \|f\|_p$ for all
$f\in V^p(g)$ and some constants $A,B>0$. This so-called stability of
translates is satisfied under a mild condition on $g$ and holds in
particular for totally positive generators $g\in L^1(\R
)$~\cite{GroechenigStoeckler2013}. The classical example is the Paley-Wiener space
\begin{equation}
    PW(\R) = \left\{f\in L^2(\R): \supp \hat{f}\subseteq [-\tfrac{1}{2},\tfrac{1}{2}\right\},
\end{equation}
with the equivalent description as a shift-invariant space $PW(\R) =
V^2(\tfrac{\sin(\pi t )}{\pi t})$. 

Analogous to the sampling theorem for bandlimited functions, which is
a topic in the theory of entire functions~\cite{seipBook}, one seeks
to understand and characterize discrete sets $X = \{x_k:k\in\Z\}\subseteq \R $ for
which 
there exist constants $0<A_p\leq B_p<\infty$ such that
\begin{equation}\label{eq:def:sampling_set}
  A_p \norm{f}_{L^p} \leq \Big( \sum _{x\in X} |f(x)|^p \Big)^{1/p} \leq  B_p  \norm{f}_{L^p}
,\qquad  \text{ for all } f\in V^p(g).
\end{equation}
Sampling in shift-invariant spaces is  useful variation of the
sampling in the Paley-Wiener space  and has become an important topic in signal
processing and of independent interest.

Although a sampling inequality~\eqref{eq:def:sampling_set} looks like
a different problem, the next lemma shows that it  is
closely connected to the theory of Gabor frames.
The following theorem is a consequence of \cite[Thm.~3.1, Thm.~3.3]{GroechenigEtAl2017} and \cite[Rem.~5.2]{GroechenigEtAl2015}.
\begin{lemma}\label{thm:Gabor_vs_Sampling}
Assume that $g$ is a continuous function with polynomial decay of
order $\sigma>1$ with  stable integer shifts.
Let $0<\alpha<1$. Then the following are equivalent:
\begin{enumerate}[(i)]
\item The Gabor system $\cG(g, \alpha,1)$ 
is a frame for $L^2(\R)$.
\item For all $x\in\R$, the set $x+\alpha\Z$ is a sampling set of
  $V^p(g)$ for some (hence for all) $p\in[1,\infty]$.
\end{enumerate}
\end{lemma}
In this context, the role of the pre-Gramian $G_{\delta } =
\big(g(k+\delta _k -l)\big)_{k,l\in\Z} $ becomes evident. Let $f(x) =
\sum _{l\in\Z} c_l g(x-l) \in V^p(g)$. Then
$$
f(k+\delta _k) = \sum _{l\in\Z} c_l g(k+\delta _k - l) = (G_\delta
c)_k \, ,
$$
so that the sampled vector $y = \big(f(k+\delta _k)\big)_{k\in \Z }$ is given by $y=
G_\delta c$.
The invertibility of $G_\delta $ implies that
$$
\Big(\sum _{k\in \Z} |f(k+\delta _k)|^p \Big) ^{1/p} = \|G_\delta c \|_p \geq
A'\|c\|_p \geq A \|f\|_p
\qquad \text{ for all }  f \in V^p(g) \, ,
$$
and likewise for the upper bound. 

In addition, since $G_\delta $ is surjective, for every $y \in \ell
^p(\Z )$ there exists a (unique)  $c\in \ell ^p(\Z)$ such that $f(k+\delta _k) =
(G_\delta c )_k = y_k$. Thus this interpolation problem is always
solvable in $\ell ^p(\Z )$. In the terminology of sampling theory, the
set $\{k+\delta _k : k\in \Z \}$ is a \emph{complete interpolating set}. 

Clearly, if a subset $\{k+\delta _k : k\in \Z \}\subseteq \R $
satisfies the  sampling inequality~\eqref{eq:def:sampling_set}, then also $x+\alpha
\Z $ satisfies \eqref{eq:def:sampling_set} (with different
constants). Thus, Proposition~\ref{prop:submatrix} (invertibility of $G_\delta $)
yields a sufficient  condition for $\mathcal{G} (g, \alpha , 1)$ to be
a Gabor frame.

In the context of shift-invariant spaces it is of interest to obtain a
general 
understanding of   complete interpolating sets for $V^p(g)$. Indeed,
in ~\cite{Gro26} it was suggested that a sharp perturbation theorem in
the style of Kadets should hold for certain shift-invariant spaces.

\subsection{A Kadets-type theorem for shift-invariant spaces}
With this background we now prove Theorem~\ref{thm:kadets} for shift-invariant
spaces with totally positive generators.

\begin{theorem}\label{thm:kadets-b}
Let $g\in L^1(\mathbb R)$ be a continuous totally positive function, and let $(x_0,\frac12)\in[0,1)^2$ be the unique zero of its Zak
transform on the unit square.
If
$\delta=(\delta_k)_{k\in\mathbb Z} \subset \R$ is a sequence satisfying
$$
\delta _{\max} :=    \sup _{k\in\Z } |\delta _k-(x_0-\tfrac{1}{2})|<\frac12, 
$$
then $X = \{ k+\delta_k : k \in \Z \}$ is a complete interpolating
sequence for $V^2(g)$.
Moreover, the constant $\frac12$ is sharp.  
\end{theorem}

\begin{proof}[Proof of Theorem \ref{thm:kadets}]
As  argued above, the statement of the theorem is equivalent to the
invertibility of  the matrix
\begin{equation}
        \preG[\delta] = (g(k+\delta_k-l))_{k,l\in\Z}
    \end{equation}
    on some  $\ell^p(\Z)$. To achieve this goal, we construct a
    continuous path from $G_\delta $ to an invertible operator and use the
    stability of the Fredholm index. The definition of such a path is inspired by Katsnelson's proof of Kadets's
theorem \cite{Kac71} and Avdonin's work  \cite{avdonin}. 

    \textbf{Step 1. Construction of an invertible operator near
      $G_\delta $.} Let $\varepsilon < 1/2 - \delta _{\max }$ and  set
    $$
    I=[x_0-1+\varepsilon,x_0-\varepsilon].
$$
Then  by assumption  $\delta_k\in I$ for all
$k\in\Z$.
Set $
\delta _* =  x_0  - 1/2$. Then $\delta _* \in I $.

Now we consider the pre-Gramian
operator $G_{\delta _*}$ based on the constant sequence $( \dots,
\delta _*, \delta _*, \dots )$.
We have seen above that $G_{\delta _*}$ is invertible if and only if
the sequence $\delta _* + \Z $ is a complete interpolating set. By
results of Janssen and Walter~~\cite{janssen94a,walter94} this is the
case if and only if the Zak transform of $g$ satisfies $Zg(\delta _*,
\xi ) \neq 0$ for all $\xi \in [0,1]$. Since $\delta _*\in I$ and
the only zeros of the Zak transform are of the form $(x_0 + k,1/2)$, this condition is
satisfied, and consequently $G_{\delta _*}$ is invertible on all $\ell
^p(\Z), 1\leq p \leq \infty $.

We note that the underlying argument is elementary: $G_{\delta _*}$
acts as the convolution operator  
$$
    (G_{\delta_*}c)_k
    =
    \sum_{l\in\mathbb Z}g(k+\delta_*-l)c_l 
    $$
with     Fourier multiplier
$
\sum _{k\in \Z } g(k+\delta _*) e^{2\pi i k\xi } =     Zg(\delta _*,
-\xi )  \neq 0 \text{ for all } \xi \in [0,1] 
$.
 Wiener's  Lemma for absolutely convergent
Fourier series then asserts that the inverse $Zg(\delta _*, \xi
)^{-1}$ has again an absolutely convergent
Fourier series, which implies that the   operator $G_{\delta
  _*}^{-1}$ is invertible on all $\ell ^p(\Z)$.

\textbf{Step 2. An intermediate family of matrices.} 
Let
$$
    \delta_k(t)=(1-t)\delta_*+t\delta_k,\qquad k\in\mathbb Z \, ,
$$
and $\delta (t) = (\delta _k(t))_{k\in \Z }$ for  $0\leq
t\leq1$. Since $I$ is convex, every sequence $\delta(t) = 
(\delta _k(t))_{k\in \Z}$ still takes all   values
in $I$.

Consequently, Proposition~\ref{prop:surjective} implies that $G_{\delta
  (t)}$ is
surjective on $\ell^1(\mathbb Z)$ for all $t\in [0,1]$. Furthermore, 
Theorem~\ref{thm:surjectivity_implies_Fredholm} implies that
$G_{\delta (t)}$ is
Fredholm on $\ell^1(\mathbb Z)$ for all $t\in [0,1]$.

 We next use 
Lemma~\ref{lem:uniformcontinuity}(iii) and see that  the path
$t\mapsto G_{\delta(t)}$ is norm-continuous as a path in
$\mathcal L(\ell^1(\mathbb Z))$. Now, the stability theorem for
Fredholm operators implies  that the Fredholm index 
$\mathrm{ind}(A)=\dim\ker A - \dim \operatorname{coker} A $ 
is constant along
this path  \cite[Thm.~5.17]{kato1980perturbation}; in particular, 
\begin{equation}\label{eq:const_ind}
    \mathrm{ind}(G_{\delta(t)}) = \mathrm{ind}(G_{\delta(0)}) = \mathrm{ind}(G_{\delta_*}), \qquad
    \text{ for all }  0\leq t\leq 1.
  \end{equation}

   To see this in detail, we observe that for every $t\in [0,1]$, there
  is an open interval $I_t$ containing $t$ such that $
  \mathrm{ind}(G_{\delta(u)})= \mathrm{ind}(G_{t})$ for $u\in
  I_t$. Since $[0,1]$ can be covered with finitely many overlapping
  intervals on which the Fredholm index is constant, 
 we find that
  $$
  \mathrm{ind}(G_{\delta }) =    \mathrm{ind}(G_{\delta (1)}) =   \mathrm{ind}(G_{\delta_*}) 
 \, .
  $$

\textbf{Step 3. Invertibility of $G_\delta $.} As we have seen in Step~2, at $t=0$, the operator $G_{\delta(0)} =
G_{\delta_*}$ is invertible and has therefore Fredholm index 
$\operatorname{ind}(G_{\delta_*}) = 0$. 

Consequently,  $G_\delta = G_{\delta(1)}$ also has Fredholm index
$$
    \operatorname{ind}(G_\delta)=0 .
$$
Since $G_\delta$ is surjective, its cokernel is zero. Thus, $\dim\ker G_\delta = \operatorname{ind}(G_\delta) = 0$ forces
$$
    \ker G_\delta=\{0\}.
$$
Hence, $G_\delta$ is invertible on $\ell^1(\mathbb Z)$. By
Theorem~\ref{thm:spec_inv}, it is invertible on all $\ell^p(\mathbb
Z)$. This means that the corresponding perturbation sequence
$\{k+\delta _k: k\in \Z\}$ is a complete interpolating set for
$V^p(g)$, $1\leq p\leq \infty $. 

\textbf{Step 4. Optimality.}
It remains to prove sharpness. If the perturbation is the constant
sequence  $\eta ,   \eta _k =  x_0$,   
then the   convolution operator 
$$
 G_\eta  c = \sum _{l\in \Z}   g(k+x_0-l) c_l
$$
has the  Fourier multiplier 
$
    \xi\mapsto Zg(x_0,-\xi).
$
Since this multiplier vanishes at $\xi=\frac12$ because
$Zg(x_0,\frac12)=0$, it is not invertible, and 
 $X= x_0+\Z$ is not a complete interpolating set.
\end{proof}

\begin{remark} {\rm 
  (i) For the Paley-Wiener space $PW(\R)$ the maximum perturbation is
  $\delta_{\max} = \tfrac{1}{4}$ by Kadets' theorem~\cite{Kadec1964}. 
Interestingly enough, $\delta_{\max} = \tfrac{1}{4}$ also holds for
the one-sided exponential  \cite{BelovEtAl2022}, which is the only
discontinuous integrable totally positive function. By contrast,  for the
shift-invariant spaces generated by an integrable \emph{continuous}
totally positive functions the maximum perturbation is  $\delta_{\max} = \tfrac{1}{2}$.  

(ii) Theorem~\ref{thm:kadets} could also be proved with limit-operator
techniques. Our approach using Theorem~\ref{thm:surjectivity_implies_Fredholm} and the stability of the
Fredholm index leads to  a slightly simpler proof.

(iii) Note that Theorem~\ref{thm:kadets} implies
Theorem~\ref{thm:main}. The particular subsequence  $\{k + \delta _k:
k\in \Z\}$ extracted from $x+\alpha \Z $ satisfies the conditions of
Theorem~\ref{thm:kadets} and therefore the associated pre-Gramian
matrix $G_\delta $ is invertible. }
\end{remark}

\section{Appendix: Lean formalization}

This appendix describes the Lean formalization of Theorem~\ref{thm:main}. It builds on Mathlib \cite{mathlib}, the mathematical library of Lean.

Theorem~\ref{thm:main} identifies the frame set of totally positive functions by establishing an equality of sets. The Lean formalization verifies the nontrivial inclusion in this equality, namely the following statement.

\begin{theorem*}
Let $g\colon\mathbb R\to\mathbb R$ be totally positive, integrable, and continuous.  If $\alpha,\beta>0$ and $\alpha\beta<1$, then the Gabor system
$$
  \bigl\{e^{2\pi i\beta n x}g(x-\alpha m):m,n\in\mathbb Z\bigr\}
$$
is a frame for $L^2(\mathbb R)$.
\end{theorem*}

The converse inclusion, namely the necessity of the condition $\alpha\beta<1$ for the frame property, is not part of the Lean formalization, since it is a well-established fact in the literature. Consequently, the formalization covers precisely the direction of Theorem~\ref{thm:main} that is new.

\subsection{Formalization of the main result}

Only a small amount of Lean syntax is needed to read the formalization of the above theorem. A general definition in Lean is introduced with \lean{def}, while a proved statement is introduced with \lean{lemma} or \lean{theorem}. In a statement, the hypotheses are listed after the name of the theorem, and each hypothesis is given a name by which it can be referred to. The nontrivial direction of the main theorem of the present paper is stated in Lean as follows.

\begin{leancode}
theorem TPFrameSet
    {α β : ℝ} (hyp1 : 0 < α) (hyp2 : 0 < β) (hyp3 : α * β < 1)
    (g : ℝ → ℝ) (hyp4 : IsTPIntegrableContinuous g) :
    IsGaborFrame g hyp4 α β
\end{leancode}

The parameters $\alpha$ and $\beta$ are declared as real numbers by \lean{{α β : ℝ}}; the curly braces mark them as implicit arguments, whose values Lean infers automatically whenever the theorem is applied. The assumptions \lean{hyp1}, \lean{hyp2} and \lean{hyp3} correspond to the hypotheses $\alpha>0$, $\beta>0$, and $\alpha\beta<1$, respectively. The function $g$ is declared by \lean{(g : ℝ → ℝ)}, and the hypothesis \lean{hyp4} combines the properties of $g$ being totally positive, integrable, and continuous. Finally, the expression after the last colon is the theorem's conclusion, namely that the Gabor system generated by $g$ with lattice parameters $\alpha$ and $\beta$ is a frame: \lean{IsGaborFrame g hyp4 α β}. The two predicates \lean{IsTPIntegrableContinuous} and \lean{IsGaborFrame} are defined in the following subsections.

\subsection{The hypotheses on the window}

The notion of total positivity is formalized directly from its definition in terms of determinants.

\begin{leancode}
def IsTP (g : ℝ → ℝ) : Prop :=
  g ≠ 0 ∧
  ∀ n : ℕ, ∀ a b : Fin n → ℝ,
    StrictMono a → StrictMono b →
    0 ≤ Matrix.det (Matrix.of fun i j => g (a i - b j))
\end{leancode}

The type \lean{Fin n} consists of the indices $0,\ldots,n-1$, so the functions
\lean{a b : Fin n → ℝ} represent two real $n$-tuples $(a_j)_{j=0}^{n-1}$ and $(b_j)_{j=0}^{n-1}$. The assumptions \lean{StrictMono a} and \lean{StrictMono b} require these tuples to be strictly increasing. In the final line, \lean{Matrix.of} interprets the function $(i,j) \mapsto g(a_i - b_j)$ as an $n \times n$ matrix, so the condition states that the determinant of the matrix $(g(a_i-b_j))_{i,j=0}^{n-1}$ is nonnegative. The classical definition moreover requires that $g$ is not constant. In the Lean definition, this appears as the condition \lean{g ≠ 0}. For the windows considered here the two requirements agree: the only integrable constant function is the zero function.

The total positivity of $g$ and the two remaining analytic assumptions are combined into a single definition.

\begin{leancode}
def IsTPIntegrableContinuous (g : ℝ → ℝ) : Prop :=
  IsTP g ∧ Integrable g ∧ Continuous g
\end{leancode}

Here \lean{Integrable g} denotes integrability with respect to the Lebesgue measure on $\R$. Thus, the hypothesis \lean{IsTPIntegrableContinuous g} asserts exactly that
$g$ is totally positive, integrable and continuous.

\subsection{Gabor atoms}

For a function $g : \R \to \R$, the Gabor atom $e^{2\pi i \beta n x} g(x-\alpha m)$ with $\alpha,\beta \in \R$ and $m,n \in \Z$ is first defined as an ordinary complex-valued function on $\R$.

\begin{leancode}
def GaborAtom_Function
    (g : ℝ → ℝ) (α β : ℝ) (m n : ℤ) : ℝ → ℂ :=
    fun x =>
    Complex.exp (2 * Real.pi * Complex.I * β * n * x) * g (x - α * m)
\end{leancode}

To use this function in the Hilbert space $L^2(\mathbb{R})$, one must first verify that it is square-integrable. This is established by the following lemma, in which \lean{volume} denotes the Lebesgue measure on $\R$.

\begin{leancode}
lemma gaborAtom_memL2
    (g : ℝ → ℝ) (hg : IsTPIntegrableContinuous g)
    (α β : ℝ) (m n : ℤ) :
    MemLp (GaborAtom_Function g α β m n) 2 volume
\end{leancode}

The predicate \lean{MemLp} expresses membership in an $L^p$-space. The conclusion of the lemma thus states that the Gabor atom is measurable and square-integrable with respect to the Lebesgue measure. The hypothesis \lean{hg} is needed here: continuity and integrability alone do not imply square-integrability. For the windows under consideration, however, the exponential decay of integrable totally positive functions implies square-integrability. The lemma is then used to construct the corresponding element of $L^2(\mathbb{R})$.

\begin{leancode}
def GaborAtom
    (g : ℝ → ℝ) (hg : IsTPIntegrableContinuous g)
    (α β : ℝ) (m n : ℤ) :
    Lp ℂ 2 (volume : Measure ℝ) :=
    (gaborAtom_memL2 g hg α β m n).toLp
\end{leancode}

The type \lean{Lp ℂ 2 volume} is Lean's version of the Hilbert space $L^2(\R) = L^2(\mathbb{R},\mathbb{C})$. The operation \lean{toLp} converts the ordinary
function, together with its proof of square-integrability, into the element of $L^2(\R)$ that it represents.

\subsection{Gabor frames}

The Gabor frame property of a totally positive, integrable, continuous function is formalized by the usual two-sided frame inequality.

\begin{leancode}
def IsGaborFrame
    (g : ℝ → ℝ) (hyp : IsTPIntegrableContinuous g)
    (α β : ℝ) : Prop :=
    ∃ A B : ℝ, 0 < A ∧ A ≤ B ∧ ∀ f : Lp ℂ 2 volume,
    let E :=
      ∑' (m : ℤ) (n : ℤ),
        ‖⟪f, GaborAtom g hyp α β m n⟫_ℂ‖ ^ 2
    A * ‖f‖ ^ 2 ≤ E ∧ E ≤ B * ‖f‖ ^ 2
\end{leancode}

The definition states the existence of frame bounds $A,B \in \R$ with $0 < A \leq B$ such that the frame inequalities hold for every $f$ in $L^2(\mathbb{R})$. Note that the predicate takes the proof \lean{hyp} as an argument, since this proof is required to form the Gabor atoms as elements of $L^2(\mathbb{R})$. The keyword \lean{let} introduces the local abbreviation \lean{E} for the double sum, in which \lean{∑'} denotes Mathlib's infinite sum and \lean{⟪f, h⟫_ℂ} denotes
the complex inner product of $f$ and $h$. In the definition above, $h$ is
the $(m,n)$-th Gabor atom given by \lean{GaborAtom g hyp α β m n}.

Consequently, the final line expresses precisely the two-sided inequality
$$
  A\| f\|^2
  \leq \sum_{m,n\in\mathbb Z}
       \bigl|\langle f,M_{\beta n}T_{\alpha m}g\rangle\bigr|^2
  \leq B\| f\|^2,
$$
where $M_{\beta n}T_{\alpha m}g(x) = e^{2\pi i \beta n x}g(x-\alpha m)$ denotes the time-frequency shifts of $g$.

\subsection{Source code}

The complete source code of the Lean formalization is available at
\begin{center}
\url{https://github.com/lukasliehr/TotallyPositive}.
\end{center}
The repository separates the statements from the proofs, so that the formalized claims can be inspected independently. The file \lean{Showcase.lean} contains the definitions and statements described above, with the two proofs, of \lean{gaborAtom_memL2} and of \lean{TPFrameSet}, replaced by the placeholder \lean{sorry}. Verifying that the formalization faithfully captures the theorem stated at the beginning of this appendix therefore only requires reading this file. The companion file \lean{Showcase_WithProofs.lean} contains the same definitions and statements, but each placeholder is replaced by a fully verified proof, whose correctness is certified by the Lean kernel during compilation.

\section{Acknowledgements}
I.~S.~was funded in part or in whole by the Austrian Science Fund (FWF)  [\href{https://doi.org/10.55776/Y1199}{10.55776/Y1199}]. For open access purposes, the authors have applied a CC BY public
copyright license to any author accepted manuscript version arising from this submission.
L.~L.~is grateful to the Azrieli Foundation for the award of an Azrieli Fellowship and acknowledges the support of this research by ISF Grant No.~854/25.

\bibliographystyle{plain}
\bibliography{bibfile}

\begin{thebibliography}{10}

\bibitem{avdonin}
S.~A. Avdonin.
\newblock On the question of {R}iesz bases of exponential functions in {$L\sp{2}$}.
\newblock {\em Vestnik Leningrad. Univ. Mat. Meh. Astronom.}, 13:5--12, 1974.

\bibitem{bannert2013discretized}
Severin Bannert, Karlheinz Gr\"ochenig, and Joachim St\"ockler.
\newblock Discretized {G}abor frames of totally positive functions.
\newblock {\em IEEE Trans. Inform. Theory}, 60(1):159--169, 2014.

\bibitem{BBG22}
Anton Baranov, Yurii Belov, and Karlheinz Gr\"{o}chenig.
\newblock Complete interpolating sequences for the {G}aussian shift-invariant space.
\newblock {\em Appl. Comput. Harmon. Anal.}, 61:191--201, 2022.

\bibitem{BelovEtAl2022}
Yurii Belov, Aleksei Kulikov, and Yurii Lyubarskii.
\newblock Irregular {G}abor frames of {C}auchy kernels.
\newblock {\em Appl. Comput. Harmon. Anal.}, 57:101--104, 2022.

\bibitem{belov2023gabor}
Yurii Belov, Aleksei Kulikov, and Yurii Lyubarskii.
\newblock Gabor frames for rational functions.
\newblock {\em Invent. Math.}, 231(2):431--466, 2023.

\bibitem{daubechies2002wavelet}
Ingrid Daubechies.
\newblock The wavelet transform, time-frequency localization and signal analysis.
\newblock {\em IEEE Trans. Inform. Theory}, 36(5):961--1005, 1990.

\bibitem{deBoorFriedlandPinkus1982}
C.~de~Boor, S.~Friedland, and A.~Pinkus.
\newblock Inverses of infinite sign regular matrices.
\newblock {\em Trans. Amer. Math. Soc.}, 274(1):59--68, 1982.

\bibitem{gabor}
D.~Gabor.
\newblock Theory of communication.
\newblock {\em J. IEE (London)}, 93(III):429--457, 1946.

\bibitem{ghosh2025gabor}
Riya Ghosh and A.~Antony~Selvan.
\newblock On {G}abor frames generated by {B}-splines, totally positive functions, and {H}ermite functions.
\newblock {\em Appl. Numer. Math.}, 207:1--23, 2025.

\bibitem{Groechenig2001}
Karlheinz Gr\"{o}chenig.
\newblock {\em Foundations of time-frequency analysis}.
\newblock Applied and Numerical Harmonic Analysis. Birkh\"{a}user Boston, Inc., Boston, MA, 2001.

\bibitem{mystery}
Karlheinz Gr{\"o}chenig.
\newblock The mystery of {G}abor frames.
\newblock {\em J. Fourier Anal. Appl.}, 20(4):865--895, 2014.

\bibitem{grosurvey22}
Karlheinz Gr\"ochenig.
\newblock Totally positive functions in sampling theory and time-frequency analysis.
\newblock In {\em Mathematical analysis, its applications and computation}, volume 385 of {\em Springer Proc. Math. Stat.}, pages 51--73. Springer, Cham, [2022] \copyright 2022.

\bibitem{Groechenig2023}
Karlheinz Gr\"{o}chenig.
\newblock Totally positive functions and {G}abor frames over rational lattices.
\newblock {\em Adv. Math.}, 427:Paper No. 109113, 12, 2023.

\bibitem{Gro26}
Karlheinz Gr\"ochenig.
\newblock Infinite totally positive matrices --- some open problems.
\newblock {\em Acta Sci. Math.}, to appear, 2026.

\bibitem{grochenig2016completeness}
Karlheinz Gr\"ochenig, Antti Haimi, and Jos\'e{}~Luis Romero.
\newblock Completeness of {G}abor systems.
\newblock {\em J. Approx. Theory}, 207:283--300, 2016.

\bibitem{GroechenigEtAl2015}
Karlheinz Gr\"ochenig, Joaquim Ortega-Cerd\`a, and Jos\'e{}~Luis Romero.
\newblock Deformation of {G}abor systems.
\newblock {\em Adv. Math.}, 277:388--425, 2015.

\bibitem{grochenig2018sampling}
Karlheinz Gr\"ochenig, Jos\'e{}~Luis Romero, and Joachim St\"ockler.
\newblock Sampling theorems for shift-invariant spaces, {G}abor frames, and totally positive functions.
\newblock {\em Invent. Math.}, 211(3):1119--1148, 2018.

\bibitem{GroechenigEtAl2017}
Karlheinz Gr\"ochenig, Jos\'e{}~Luis Romero, and Joachim St\"ockler.
\newblock Sampling theorems for shift-invariant spaces, {G}abor frames, and totally positive functions.
\newblock {\em Invent. Math.}, 211(3):1119--1148, 2018.

\bibitem{GroechenigStoeckler2013}
Karlheinz Gr\"ochenig and Joachim St\"ockler.
\newblock Gabor frames and totally positive functions.
\newblock {\em Duke Math. J.}, 162(6):1003--1031, 2013.

\bibitem{heil2007history}
Christopher Heil.
\newblock History and evolution of the density theorem for {G}abor frames.
\newblock {\em J. Fourier Anal. Appl.}, 13(2):113--166, 2007.

\bibitem{janssen2002some}
A.~J. E.~M. Janssen.
\newblock Some counterexamples in the theory of {W}eyl-{H}eisenberg frames.
\newblock {\em IEEE Trans. Inform. Theory}, 42(2):621--623, 1996.

\bibitem{Janssen1996}
A.~J. E.~M. Janssen.
\newblock Some {W}eyl-{H}eisenberg frame bound calculations.
\newblock {\em Indag. Math. (N.S.)}, 7(2):165--183, 1996.

\bibitem{janssen2002hyperbolic}
A.~J. E.~M. Janssen and Thomas Strohmer.
\newblock Hyperbolic secants yield {G}abor frames.
\newblock {\em Appl. Comput. Harmon. Anal.}, 12(2):259--267, 2002.

\bibitem{janssen94a}
A.J.E.M. Janssen.
\newblock The {Z}ak transform and sampling theorems for wavelet subspaces.
\newblock {\em IEEE Transactions on Signal Processing}, 41(12):3360--3364, 1993.

\bibitem{Kac71}
V.~\`E. Kacnel'son.
\newblock Bases of exponential functions in {$L\sp{2}$}.
\newblock {\em Funkcional. Anal. i Prilo\v zen.}, 5(1):37--47, 1971.

\bibitem{Kadec1964}
M.~I. Kadec.
\newblock The exact value of the {P}aley-{W}iener constant.
\newblock {\em Dokl. Akad. Nauk SSSR}, 155:1253--1254, 1964.

\bibitem{karlin1968total}
Samuel Karlin.
\newblock {\em Total positivity. {V}ol. {I}}.
\newblock Stanford University Press, Stanford, CA, 1968.

\bibitem{kato1980perturbation}
Tosio Kato.
\newblock {\em Perturbation theory for linear operators}.
\newblock Classics in Mathematics. Springer-Verlag, Berlin, 1995.
\newblock Reprint of the 1980 edition.

\bibitem{Klo15}
Tobias Kloos.
\newblock Zeros of the {Z}ak transform of totally positive functions.
\newblock {\em J. Fourier Anal. Appl.}, 21(5):1130--1145, 2015.

\bibitem{KS14}
Tobias Kloos and Joachim St\"{o}ckler.
\newblock {Z}ak transforms and {G}abor frames of totally positive functions and exponential {B}-splines.
\newblock {\em J. Approx. Theory}, 184:209--237, 2014.

\bibitem{kloos2016implementation}
Tobias Kloos, Joachim St\"ockler, and Karlheinz Gr\"ochenig.
\newblock Implementation of discretized {G}abor frames and their duals.
\newblock {\em IEEE Trans. Inform. Theory}, 62(5):2759--2771, 2016.

\bibitem{lindner2006infinite}
Marko Lindner.
\newblock {\em Infinite matrices and their finite sections}.
\newblock Frontiers in Mathematics. Birkh\"auser Verlag, Basel, 2006.
\newblock An introduction to the limit operator method.

\bibitem{lindner2014affirmative}
Marko Lindner and Markus Seidel.
\newblock An affirmative answer to a core issue on limit operators.
\newblock {\em J. Funct. Anal.}, 267(3):901--917, 2014.

\bibitem{lyubarskiiframes}
Yurii Lyubarskii.
\newblock Frames in the {B}argmann space of entire functions.
\newblock In {\em Entire and subharmonic functions}, volume~11 of {\em Adv. Soviet Math.}, pages 167--180. Amer. Math. Soc., Providence, RI, 1992.

\bibitem{neumann}
J.~von Neumann.
\newblock {\em Mathematische {G}rundlagen der {Q}uantenmechanik}.
\newblock Springer, Berlin, 1932.
\newblock English translation: ``Mathematical foundations of quantum mechanics,'' Princeton Univ. Press, 1955.

\bibitem{pinkus2010totally}
Allan Pinkus.
\newblock {\em Totally positive matrices}, volume 181 of {\em Cambridge Tracts in Mathematics}.
\newblock Cambridge University Press, Cambridge, 2010.

\bibitem{rabinovitch2004limit}
Vladimir Rabinovich, Steffen Roch, and Bernd Silbermann.
\newblock {\em Limit operators and their applications in operator theory}, volume 150 of {\em Operator Theory: Advances and Applications}.
\newblock Birkh\"auser Verlag, Basel, 2004.

\bibitem{Schoenberg1947}
I.~J. Schoenberg.
\newblock On totally positive functions, {L}aplace integrals and entire functions of the {L}aguerre-{P}olya-{S}chur type.
\newblock {\em Proc. Nat. Acad. Sci. U.S.A.}, 33:11--17, 1947.

\bibitem{Schoenberg1951}
I.~J. Schoenberg.
\newblock On {P}\'olya frequency functions. {I}. {T}he totally positive functions and their {L}aplace transforms.
\newblock {\em J. Analyse Math.}, 1:331--374, 1951.

\bibitem{seidel2015semi}
Markus Seidel.
\newblock On semi-{F}redholm band-dominated operators.
\newblock {\em Integral Equations Operator Theory}, 83(1):35--47, 2015.

\bibitem{seip-wallsten}
K.~Seip and R.~Wallst{\'e}n.
\newblock Density theorems for sampling and interpolation in the {B}argmann-{F}ock space. {I}{I}.
\newblock {\em J. Reine Angew. Math.}, 429:107--113, 1992.

\bibitem{Seip1992}
Kristian Seip.
\newblock Density theorems for sampling and interpolation in the {B}argmann-{F}ock space. {I}.
\newblock {\em J. Reine Angew. Math.}, 429:91--106, 1992.

\bibitem{seipBook}
Kristian Seip.
\newblock {\em Interpolation and sampling in spaces of analytic functions}, volume~33 of {\em University Lecture Series}.
\newblock American Mathematical Society, Providence, RI, 2004.

\bibitem{ShinSun2009}
Chang~Eon Shin and Qiyu Sun.
\newblock Stability of localized operators.
\newblock {\em J. Funct. Anal.}, 256(8):2417--2439, 2009.

\bibitem{Tessera2010}
Romain Tessera.
\newblock Left inverses of matrices with polynomial decay.
\newblock {\em J. Funct. Anal.}, 259(11):2793--2813, 2010.

\bibitem{mathlib}
{The mathlib {C}ommunity}.
\newblock The {L}ean {M}athematical {L}ibrary.
\newblock In {\em Proceedings of the 9th {ACM} {SIGPLAN} International Conference on Certified Programs and Proofs}, CPP 2020, New Orleans, LA, USA, January 2020. ACM.

\bibitem{VinogradovUlitskaya2017}
O.~L. Vinogradov and A.~Yu. Ulitskaya.
\newblock Zeros of the {Z}ak transform of averaged totally positive functions.
\newblock {\em J. Approx. Theory}, 222:55--63, 2017.

\bibitem{walter94}
Gilbert~G. Walter.
\newblock {\em Wavelets and other orthogonal systems with applications}.
\newblock CRC Press, Boca Raton, FL, 1994.

\bibitem{zeeviZibulskiPorat1998multiwindow}
Yehoshua~Y. Zeevi, Meir Zibulski, and Moshe Porat.
\newblock Multi-window {G}abor schemes in signal and image representations.
\newblock In {\em Gabor analysis and algorithms}, Appl. Numer. Harmon. Anal., pages 381--407. Birkh\"auser Boston, Boston, MA, 1998.

\bibitem{zibulskiZeevi1997analysis}
Meir Zibulski and Yehoshua~Y. Zeevi.
\newblock Analysis of multiwindow {G}abor-type schemes by frame methods.
\newblock {\em Appl. Comput. Harmon. Anal.}, 4(2):188--221, 1997.

\end{thebibliography}

\end{document}